\documentclass[hidelinks,onefignum,onetabnum]{siamart251216}

\usepackage{lipsum}
\usepackage{amsfonts}
\usepackage{graphicx}
\usepackage{epstopdf}
\usepackage{algorithmic}
\ifpdf
  \DeclareGraphicsExtensions{.eps,.pdf,.png,.jpg}
\else
  \DeclareGraphicsExtensions{.eps}
\fi

\usepackage{bm}
\usepackage{bbm}
\usepackage{nccmath} 
\usepackage{tikz}
\usepackage{wrapfig}
\usetikzlibrary{shapes.misc}
\usepackage{smartdiagram}
\let\oldxi\xi
\usesmartdiagramlibrary{additions}

\usepackage{array}
\usepackage{thmtools}
\usepackage{euscript}
\usepackage{longtable} 
\newcommand{\E}{\mathbb{E}}
\newcommand{\N}{\mathbb{N}}
\newcommand{\R}{\mathbb{R}}

\newcommand{\balpha}{\bm{\alpha}}
\newcommand{\oalpha}{|\balpha|}
\newcommand{\bbeta}{\bm{\beta}}
\newcommand{\obeta}{|\bm{\beta|}}
\newcommand{\bgamma}{\bm{\gamma}}
\newcommand{\ogamma}{|\bgamma|}

\newcommand{\A}{\bm{A}}
\newcommand{\B}{\bm{B}}

\newcommand{\x}{\bm{x}}
\newcommand{\bxi}{\bm{\xi}}
\newcommand{\presumneg}{\hspace{-4pt}}
\newcommand{\rvline}{\hspace*{-\arraycolsep}\vline\hspace*{-\arraycolsep}}
\newcommand{\trip}{\langle H_{\balpha}, H_{\bbeta}, H_{\bgamma} \rangle}

\newsiamremark{remark}{Remark}
\newsiamremark{hypothesis}{Hypothesis}
\crefname{hypothesis}{Hypothesis}{Hypotheses}
\newsiamthm{claim}{Claim}
\newsiamremark{fact}{Fact}
\crefname{fact}{Fact}{Facts}
\newsiamremark{example}{Example}
\headers{Triple Products of Correlated Hermite Polynomials}{L. Lyman}

\title{Triple Products of Correlated Multivariate Hermite Polynomials\thanks{Submitted to the editors on Nov.~12 2026.}}

\author{Laura A. Lyman\thanks{Department of Mathematics and Statistics, Mount Holyoke College, South Hadley, MA (\email{lalyman@mtholyoke.edu}).}}

\usepackage{amsopn}

\ifpdf
\hypersetup{
  pdftitle={Triple Products of Correlated Multivariate Hermite Polynomials},
  pdfauthor={L. Lyman}
}
\fi

\begin{document}

\maketitle

\begin{abstract}
Stochastic Galerkin schemes (SGS) are used to approximate solutions to differential equations with stochastic parameters. In particular, they express the solution with an infinite orthogonal functional basis of the parameter space and derive a deterministic surrogate system of PDEs whose solution approximates that of the original stochastic problem. When the parameters are correlated, the infinite series representation of the solution does not provably converge for all joint distributions; further, classic decorrelating transformations (e.g., Rosenblatt, Nataf) often introduce known numerical instabilities. Rahman (2017) established convergence in probability and distribution for correlated Gaussian parameters when the basis functionals are multivariate Hermite polynomials (MHPs), where these polynomials are only weakly orthogonal with respect to the underlying joint density. Due to weak orthogonality, the $N$th moments of MHPs that emerge during Galerkin projection do not trivially decompose into products of one-dimensional expectations via tensorization. Namely, an $N$th-order moment has no closed-form expression for $N > 2$ and consequently is estimated numerically. Because the number of such moments grows combinatorially with parameter dimension and polynomial degree, approximating each one individually via sampling or quadrature can quickly become cost-prohibitive. In this paper, we introduce and prove several foundational properties of  MHPs. These results culminate in an explicit formula for evaluating the $N = 3$ moments (triple products) directly without numerical estimation. 
\end{abstract}

\begin{keywords}
Uncertainty quantification (UQ), stochastic Galerkin schemes (SGS), polynomial chaos,
multivariate Hermite polynomials
\end{keywords}

\begin{MSCcodes}
65N30, 35R60, 65C20, 65K20
\end{MSCcodes}

\section{Introduction}
This paper considers a subset of \emph{forward propagation problems}, in which the goal is to quantify the variability in some output \emph{quantity of interest} (QoI) given uncertain model parameters represented by random variables. The random variables are prescribed a joint probability density function, which is selected by the scientist or engineer based on physical knowledge of the problem. Using the notation in \cref{fig_forward_prop}, we can imagine that the model functional $\mathcal{F}$ represents a  system of partial differential equations (PDEs) with solution $u$, where the uncertain input parameters appear in initial or boundary conditions. Accordingly, the response $u$ is a stochastic PDE solution, since it is a function of the random parameters. 

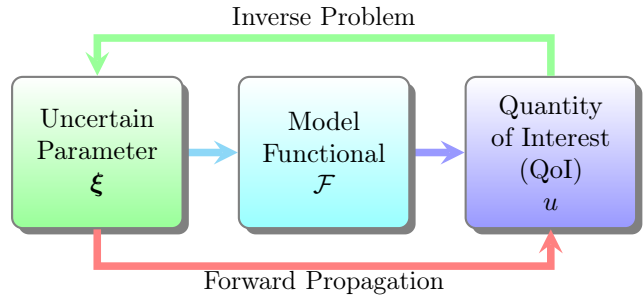
\begin{figure}[H]
\begingroup
\setlength{\intextsep}{0pt}%
\centerline{\scalebox{0.9}{Inverse Problem}} \smallskip
\centerline{
\smartdiagramset{module x sep= 3cm,
back arrow disabled = false,
module minimum width=2cm, module minimum height = 2cm, text width=2cm,
set color list={green!40}
}
\smartdiagramadd[flow diagram:horizontal]{Uncertain Parameter \\ $\bm{\oldxi}$, Model \\ Functional $\mathcal{F}$, Quantity of Interest (QoI) \\ ${{u}}$}{}
\begin{tikzpicture}[overlay]
        \draw[additional item arrow type,color=red!50] (module1) -- ++(0,-1.5) -| (module3);
    \end{tikzpicture}
}
\medskip
\vspace{5pt}
\centerline{\scalebox{0.9}{Forward Propagation}} \smallskip
\caption{Framework for uncertainty propagation problems.}\label{fig_forward_prop}
\endgroup
\vspace{-5pt}
\end{figure}

Several methods have been developed within the field of uncertainty quantification\footnote{The phrase \emph{uncertainty quantification} emerged within engineering communities to describe the interface between statistics and physics-related models with (often) expensive-to-evaluate simulations. As noted by Berger and Smith \cite{UQ_name_history}, ``In statistics, this name appears redundant, as the primary role of statistics is uncertainty quantification; in economics, the phrase is an oxymoron, as Knightian uncertainty is by definition unquantifiable." Nevertheless, UQ has become the standard term for this interface, so we follow the custom as~well.} (UQ) to tackle the forward propagation of uncertainty, especially in the context of solving stochastic systems in engineering applications \cite{constantine_thesis:master, Debusschere2017, Mani_paper, stoch_galer_2:master, GI_book, GI_Roe,schober2017probabilistic}. As illustrated in \cref{fig_pillars_UQ}, these methods in UQ can be classified into three broad categories: sampling methods (\cref{stat_methods}), quadrature and stochastic collocation, and spectral Galerkin methods (\cref{sec_intro_PC}). While this paper does not discuss quadrature and stochastic collocation, there is a wealth of research and literature on the subject \cite{Gunzburger2017, KL_ref_from_navarro, Smith, xiu_stochcollsurvey}.

 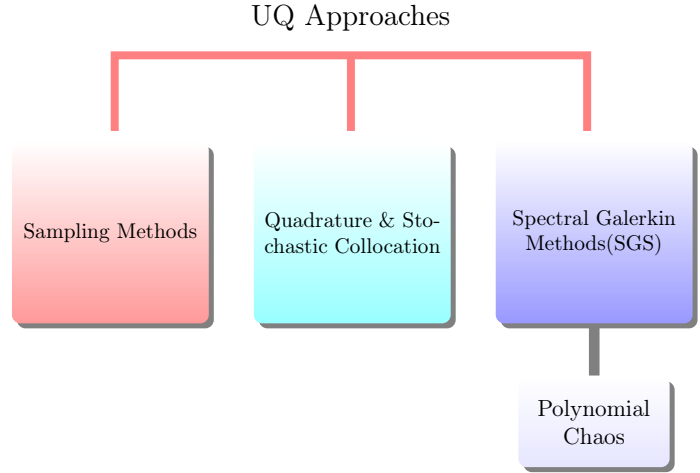
\begin{figure}[H]
{\smartdiagramset{module x sep= 3cm,
back arrow disabled = true,
module minimum width=2cm, module minimum height = 2cm, text width=2cm, uniform arrow color=true, arrow color=white, set color list={red!40} }
\setlength{\intextsep}{0pt}%
\bigskip
\bigskip
\vspace{20pt}
\centerline{
\scalebox{.7}{\smartdiagramset{border color=none, module x sep= 4cm, arrow line width = 0cm, circular final arrow disabled=true, module minimum width=3cm, module minimum height = 3cm, text width=3cm, arrow style=, arrow tip=,
additions={
additional item offset=0.85cm,
additional arrow tip=stealth,
additional arrow line width=5pt,
additional item bottom color=blue!10,
additional item text width = 2cm,
additional item width = 2.5cm,
additional arrow tip=to,
additional item height = 1.5 cm, additional item shadow=drop shadow
}}
\smartdiagramadd[flow diagram:horizontal]{Sampling Methods, Quadrature~\& Stochastic Collocation, Spectral Galerkin Methods (SGS)}
{
below of module3/Polynomial Chaos
}
\smartdiagramconnect{-}{module3/additional-module1}
}
}
\begin{tikzpicture}[overlay]
        \draw[color=red!50, line width = 3pt] (3.375,3) -- (3.375,4) -- (9.625,4) -- (9.625, 3);
        \draw[color=red!50, line width = 3pt] (6.5,4) -- (6.5,3) ;
        \node at (6.5,4.5) {UQ Approaches};
   \end{tikzpicture}
}
\vspace{40pt}
\caption{The three pillars of UQ for analyzing forward propagation problems.}\label{fig_pillars_UQ}
\end{figure}


 \subsection{Sampling methods}\label{stat_methods}
 The sampling approach uses many realizations of an input random variable to determine distributional information about the output solution of interest. The most common example is a simple Monte Carlo (MC) simulation. The advantages of simple MC are (1) it is generally straight-forward to implement, and (2) its accuracy is not dependent on the dimensionality of the uncertain input parameter (but instead on the number of parameter realizations) \cite{sampling_ref, mcbook:master}. However, by the law of large numbers,  MC algorithms  suffer from a slow convergence rate proportional to the inverse square root of the number of samples \cite{ mypaper:master, GI_book, Smith}.  Accordingly, if a single sample evaluation of the output solution of interest is time consuming, and/or a large number of sample evaluations is required to converge within a desired accuracy level, MC simulations can become prohibitively expensive. Further, if the goal is to assess the behavior of the stochastic solution given low-probability events, sampling methods are not ideal, since they will require a large number of realizations \cite{xiu_thesis:master}. 
 
To accelerate convergence, several techniques have been developed, such as Latin Hypercube Sampling,  Quasi-Monte Carlo (QMC), and Markov Chain Monte Carlo (MCMC) \cite{MCMC_2, MCMC_1, Smith}; for an excellent overview, see \cite{mcbook:master, practicalqmc}. These methods can improve the efficiency of the brute-force Monte Carlo method --- though additional restrictions are imposed based on their specific designs, meaning their applicability may be limited in some contexts \cite{xiu_thesis:master}. 

 \subsection{Stochastic Galerkin schemes (SGS)}\label{sec_intro_PC}
For MC algorithms, if each sample evaluation is expensive --- as is often the case for a PDE solution underlying a large-scale simulation --- the slow convergence can make obtaining tens of thousands of evaluations computationally infeasible \cite{constantine_thesis:master, Feinberg, Navarro}. These computational constraints are worsened when the resolved model requires a high degree of accuracy. In such cases, stochastic Galerkin methods might be better suited to the problem at hand. 

Stochastic Galerkin schemes (SGS) are used to approximate the solution to parametrized differential equations. In particular, they utilize an infinite functional basis of the parameter space to express the solution and derive a \emph{deterministic} system of PDEs, which can then be solved with standard numerical techniques \cite{constantine_primer:master, UQHandbook_IntroUQ, Smith}. The deterministic coefficients in the functional series representation of the solution are sometimes called \emph{random modes}. Since Galerkin methods are based on projecting the residual onto the space of approximating polynomials, its accuracy is optimal in an $L^2$ sense \cite{Smith}. SGS are part of a broader class known as \emph{spectral methods} \cite{alexanderian2013:master,constantine_thesis:master, UQHandbook_IntroUQ, ghanem_and_spanos:master, le_maitre:master, GI_book}.

SGS are called \emph{intrusive} in the sense that a \emph{new} system associated with the projection is developed and implemented, where the new system requires constructing expectations or inner products not typically used in deterministic codes; therefore, existing codes cannot be used directly to recover the random modes \cite{Debusschere2017, GI_book,Smith}. In other words, intrusive methods for uncertainty quantification reformulate the original stochastic model equations to obtain a system of deterministic equations --- the trade-off being that the deterministic system involves more equations and variables than the original model. For applications with large legacy codes or `black-box' executable files, a non-intrusive approach (e.g., Monte Carlo, stochastic collocation/quadrature \cite{xiu_stochcollsurvey}, etc.) might be more appropriate \cite{constantine_thesis:master}. 

In the conventional approach for SGS, the functional basis elements are chosen to be polynomials that are orthogonal with respect to the joint density of the random input variables. This representation has its roots in the work of Wiener \cite{wiener:master}, who expressed a Gaussian process as a countably infinite series of Hermite polynomials. Ghanem and Spanos \cite{ghanem_and_spanos:master} truncated Wiener’s representation and used the resulting finite series as a key ingredient in a stochastic finite element method. SGS based on polynomial expansions are often referred to as \emph{polynomial chaos} approaches.

When compared to MC algorithms, initial spectral method applications to UQ problems showed \emph{orders-of-magnitude} reductions in the cost needed to estimate statistics with comparable accuracy \cite{xiu_karniadakis:master}; however, there are of course some exceptions, e.g., see \cite{exceptionsToPCefficiency1, exceptionsToPCefficiency2}. Several PCE extensions have been developed in the last years  \cite{Navarro} that have shown to be effective in a large number of applications in different fields: fluid dynamics \cite{Feinberg,le_maitre:master, verif_valid_ref_2, GI_book}, air-water flows in soil \cite{airflowsoil}, aircraft design \cite{ZDR}, radiation therapy \cite{Perko_2016, proton_therapy_thesis}, etc. Further, spectral Galerkin methods have the advantage of returning the entire distribution of the output QoI, from which any descriptive statistics (mean, variance, etc.) can be computed directly; this is in direct contrast to sampling strategies, which only compute pre-specified descriptive statistics.

\subsection{Limitations of polynomial chaos}\label{section_pc_limits}
Classically, polynomial chaos methods have the following disadvantages. 

\subsubsection*{Distribution of stochastic input parameter}
In this paper, we assume that the stochastic random variable input $\xi$ is Gaussian. However, in a series of papers  \cite{gPC_tech_report, xiu_thesis:master, other_gPC_paper, xiu_karniadakis:master}, Xiu et al.~introduced spectral expansions for non-Gaussian $\xi$, where the corresponding basis polynomials are selected to be orthogonal with respect to the density of $\xi$. For example, Legendre polynomials are selected when $\xi \sim U(a,b)$, Laguerre polynomials when $\xi \sim \Gamma(a,b)$, etc. This spearheaded the use of the so-called \emph{generalized polynomial chaos expansion} (gPCE) in an explosion of research applications \cite{const_other_ref, gPCE_research_ex_1,GI_Roe}. 

With regard to the convergence of these generalized expansions, the issue was approached rather guardedly \cite{ernst:master}; one of the earlier papers \cite{other_gPC_paper} on the topic states, ``Convergence to second-order stochastic processes can be possibly obtained as a generalization of Cameron–Martin theorem.'' Only in 2010 did Ernst et al.~prove that an arbitrary random variable with finite variance can be expressed as a generalized polynomial chaos expansion \emph{only} if the underlying probability measure is uniquely determined by its moments \cite{ernst:master} --- a condition which fails in some simple cases, such as when $\xi$ has a lognormal distribution (see pages 106 - 107 of \cite{prob_book}). Restrictions like this preclude the use of a polynomial chaos expansion for general densities constructed using Bayesian techniques~\cite{Smith}.

\subsubsection*{Poor convergence for non-smooth solutions}  
As noted by Xiu et al., the exponential convergence rate of the PCE degrades when mapping the random input vector to the response is not sufficiently smooth \cite{xiu_thesis:master, xiu_karniadakis:master} due to Gibbs phenomenon \cite{Gibbs_orig, Gibbs_anotherref}. Further, Gibbs oscillations may cause a spectral approximation to seemingly attain nonphysical values. In the worst case, the PCE may fail to capture essential features of the stochastic quantity of interest \cite{LEMAITRE,GI_book}. Consequently, the PDEs considered in this paper are restricted to those whose solutions are smooth within the domain of interest. 

\subsubsection*{Assumed independence of stochastic parameters}
For systems that depend on multiple stochastic parameters, we can extend the univariate versions of the Wiener-Hermite PCE, the gPCE, and stochastic Galerkin schemes (SGS) in general. In the multivariate case, one of the \emph{essential} assumptions when applying SGS is that the stochastic model parameters are independent. This assumption allows the joint distribution of the random variables to decompose as a product of probability densities, enabling the easy construction of multivariate orthogonal polynomials via tensorizing the spaces of univariate orthogonal polynomials \cite{rahman:master, Smith, xiu_thesis:master}. These orthogonal basis functionals are consistent with the underlying measure \cite{ernst:master} and thus constitute a complete Hilbert basis \cite{rahman:master}; as clarified only recently \cite{ernst:master}, this completeness\footnote{The completeness of the Hilbert basis is equivalent to the underlying probability measure being uniquely determined by its moments \cite{ernst:master}, which is the condition discussed in \cref{section_pc_limits}.} is vitally important for convergence of the truncated PCE to the correct limit.
 
\subsubsection*{Curse of dimensionality}
As with several numerical techniques, polynomial chaos suffers from the \emph{curse of dimensionality}, which is to say that as the number of parameters increases, the work required to compute the approximation increases combinatorially \cite{Debusschere2017, Feinberg, rahman:master, Smith, xiu_thesis:master}. Over the years, there has been a vast amount of literature\footnote{For a survey of this work, see Giebel~\cite{griebel_2006}.} on methods for alleviating this curse of dimensionality for certain classes of problems, several of which involve differential equations with stochastic inputs \cite{constantine_thesis:master}. Even so, PC methods are typically used to address low-dimensional problems --- that is, problems for which the number of parameters is not very large (say, less than ten) \cite{rahman:master}.

\subsection{Multivariate polynomial chaos}\label{multi_polychaos_intro}
As described in \cref{section_pc_limits},  in the multivariate case, the PCE (along with the gPCE and SGS overall) is \emph{founded} on the assumption that its input random variables are independent \cite{constantine_thesis:master, Feinberg, Navarro, GI_book, Smith}. 
 
\subsubsection*{Decorrelating transformations}
In reality, there may exist significant correlation or dependence among input variables.\footnote{For a few of numerous examples of modeling PDEs with correlated stochastic Gaussian parameters, see (8.22) and (8.24) in \cite{Smith}, which examine the steady state heat model and an HIV model.} One approach is to apply a  decorrelating transformation \cite{Feinberg, Navarro, Smith} and work with the PCE in the resulting independent\footnote{As mentioned, we focus on Gaussian parameters and therefore may use the words \emph{uncorrelated} and \emph{independent} interchangeably.} parameters; specifically, some options involve linear decorrelating maps (e.g., Rosenblatt \cite{Rosenblatt}, Nataf \cite{Nataf}), the K-L expansion \cite{KL_ref_from_navarro}, or proper orthogonal decomposition \cite{orthog_decomp_ref}. 
 
 The challenge is that such transformations tend to introduce discontinuities or steep gradients \cite{Feinberg, GI_book, xiu_karniadakis:master}, which (as discussed in \cref{section_pc_limits}) may cause a significant slowdown in the convergence speed of PC methods due to Gibbs phenomenon \cite{Gibbs_orig, Gibbs_anotherref, GI_Roe}. Further, for applications like general nonlinear conservation laws (e.g., fluid dynamics problems) \cite{GI_Roe}, diffusion modeling in multilayered materials \cite{Feinberg}, and in blood flow simulations at arterial junctions \cite{cardiovasc_ex}, the PDE solution may have discontinuities and/or steep gradients to \emph{begin with}  --- so these effects on convergence are worsened by decorrelating; the result is an expansion with a lower convergence rate as the polynomial degree increases \cite{const_other_ref, Feinberg, Gibbs_orig, Gibbs_anotherref, xiu_thesis:master,xiu_karniadakis:master}. Still, utilizing such transformations is an active area of research in the context of polynomial chaos \cite{Feinberg,Jakeman,DataDrivenUsesDecorrelateTransformation}. 
 
 The work in this paper does not use transformation-based methods.

\subsubsection*{Transformation-free approaches}
There have been several transformation-free approaches to account for arbitrary but dependent Gaussian random variables. Using Gram–Schmidt orthogonalization \cite{Golub}, Navarro et al.~discussed construction of a PCE with multivariate orthogonal polynomials for correlated variables \cite{Navarro}. However, the existence of multivariate Hermite orthogonal polynomials, which can be used as a basis for dependent Gaussian measures, has not been recognized \cite{rahman:master}. Soize and Ghanem proposed orthogonal functional bases with regard to a general dependent probability measure of random input, but the bases elements are not necessarily \emph{polynomials} \cite{Soize_Ghanem}; as a result, an analytical treatment of PCE-based statistics in this case is highly non-trivial \cite{Smith}. Furthermore, both the works of Navarro and Soize/Ghanem do not address a \emph{fundamental} concern raised by Ernst et al.~about the completeness of their orthogonal basis functionals, which was shown to be a necessary condition for convergence \cite{ernst:master, rahman:master}. 

In 2017, Rahman constructed a PCE with \emph{weakly} orthogonal (defined in \cref{weak_orthog}) basis functionals consistent with the Gaussian measure, which was shown to converge to the correct limit in mean-square, probability, and distribution  \cite{rahman:master}. We will refer to this series expansion as the \emph{correlated multivariate PCE} (or \emph{CPCE})\footnote{In \cite{rahman:master}, Rahman refers to the introduced countably infinite PC series in correlated Gaussians as the \emph{generalized PCE}. However, this language might cause confusion, given that the phrase \emph{generalized PCE} (gPCE) typically refers to the univariate PC expansion first introduced by Xiu and Karniadakis in 2002 for non-Gaussian random variables \cite{xiu_thesis:master, xiu_karniadakis:master}; see the discussion in \cref{section_pc_limits}.} and use it exclusively throughout the rest of this document.

\subsection{Higher-order products}\label{higher_order_products}

In the Galerkin framework, one projects weighted residuals onto a finite-dimensional subspace spanned by basis functionals to provide the constraints required to solve for the deterministic coefficients \cite{Smith}. When utilizing the CPCE, the projection step introduces expectations of products of multivariate Hermite polynomials in the Gaussian parameters. One major consequence of weak orthogonality is that these \emph{$N$th products} of the form $\mathbb{E}_{\bxi}(H_{\balpha^{(1)}}(\bxi) \cdots H_{\balpha^{(N)}}(\bxi))$ no longer decompose into a multiplication of their one-dimensional counterparts. Often these expectation integrals over the $n$-dimensional parameter space are approximated numerically  --- and doing so can be highly nontrivial \cite{constantine_thesis:master,Feinberg,Smith}. Typically, these integrals are estimated with some combination of sampling methods, standard Gaussian quadrature/collocation, and/or sparse grid techniques (when possible) \cite{alexanderian2013:master, Gunzburger2017, rahman:master, intro_ref:master, xiu_thesis:master, xiu_stochcollsurvey} --- that is, with the statistical methods discussed in \cref{stat_methods}. In 2017, Rahman first provided a formula for the \emph{double products} ($N = 2$) or \emph{second-order moments}, $\mathbb{E}_{\bxi}(H_{\balpha}(\bxi) H_{\bbeta}(\bxi)),$ that appear during Galerkin projection \cite{rahman:master}. Prior to the work given in this paper, in the presence of multiple correlated uncertainties, there was no closed-form expression for the $N$th products of the multivariate Hermite polynomials for any $N$ larger than two. 

\subsection{Summary of contributions}
This paper presents technical contributions regarding multivariate Hermite polynomials (MHPs) whose inputs are correlated Gaussian random variables, culminating in a formula for their triple products.~We assume $\bm{\xi}~\sim~\mathcal{N}(\bm{0}, \Sigma)$ for known (SPD) covariance matrix ($\Sigma$). All expected values whose inputs involve MHPs are understood to mean expectation with respect to the $d\bxi$ measure. The key results are outlined below. 
\begin{itemize}
\item An MHP indexed by a sum (\cref{sum_of_indices}) can be decomposed as a sum of lower-order MHPs.
\item The higher-order derivative $\partial^{\A}(H_{\bbeta})$ of an MHP can be expressed in terms of double products and lower-order Hermite polynomials (\Cref{deriv_formula}), which relies on \cref{sum_corollary}.
\item There is a formula for the triple product (\Cref{trip_prod}) in terms of double products. For  $\balpha, \bbeta, \bgamma \in \N_0^n$ such that $\ogamma = \text{max}(\oalpha, \obeta, \ogamma)$,
\begin{align*}
\langle H_{\balpha}, H_{\bbeta}, H_{\bgamma} \rangle =  \hspace{-10pt} \sum_{\substack{\A \thinspace \le \balpha \\ \B \le \bbeta \\ |\A| = |\B| = k}} \mbinom{\balpha}{\A} \mbinom{\bbeta}{\B} \langle H_{\A}, H_{\B} \rangle \langle H_{\bm{\alpha + \beta - A - B}}, H_{\bgamma} \rangle.
\end{align*}
for $k := \frac{\oalpha + |\bbeta| - |\bgamma|}{2}.$ Further, the triple product is shown to equal zero when certain conditions on the relative sizes and orders of the mult-indices are satisfied; namely, if 
\begin{enumerate}
\item[(a)] $\oalpha + \obeta < \ogamma$ or 
\item[(b)] $\oalpha + \obeta + \ogamma$ is odd,
\end{enumerate}
then the triple product $\langle H_{\balpha}, H_{\bbeta}, H_{\bgamma} \rangle$ = 0. 

 \item For one-dimensional Hermite polynomials and indicies $i,j,k \in \N_0$, the expression above reduces to the well-known triple product formula (\Cref{one_d_trip_form}).

\end{itemize}

The paper is organized as follows. \Cref{defs_and_notation} presents relevant definitions and notational conventions used throughout the document along with necessary background on correlated multivariate Hermite polynomials (MHPs). In \cref{sec_PC_corr_gauss}, these polynomials serve as the basis functionals for the so-called \emph{correlated polynomial chaos expansion} (CPCE). \Cref{example_application} presents an example application of solving a stochastic variant of Burgers' equation with the CPCE and helps illustrate the practical value of a triple product formula that a computer can evaluate directly. In \cref{multivar_formulae}, we introduce and prove several useful decompositions for MHPs and their derivatives. Then in \cref{trip_prod_sec}, we present the paper's central result: a closed-form expression for the triple product in terms of double products; afterwards, some small examples of the formula's usage are provided. \Cref{sect_trip_prod_n_1} demonstrates how the expression for the triple product of correlated MHPs reduces to a well-known one-dimensional formula when setting $n = 1.$ Next, \cref{simple_MC_conv} contains a numeric example in which $35 + 165 = 200$ total\footnote{There are actually 56 triple products for the $n = 2$ case and 220 for the $n = 3$ case, giving 56 + 220 = 276 total; however, any products involving $H_{\textbf{0}} = 1$ were excluded since they reduced to lower-order moments (that is, single products that equal 0 or double products with an existing evaluation formula).} third-order moments are approximated via simple Monte Carlo, from which we can observe the estimates converging to their predicted formula values as the number of sample realizations becomes sufficiently large. Concluding remarks are provided in \cref{sec:conclusions}.


 \section{Definitions, notation, and mathematical preliminaries}
 \label{defs_and_notation}

\begin{definition}[Multi-index]\label{multi-index} An $n$-dimensional \textit{multi-index} $\balpha$ is an $n$-tuple $(\alpha_1, \ldots, \alpha_n)~\in~\N_0^n$ of non-negative integers. Each $\alpha_i$ is referred to as the \textit{$i$th element of the multi-index~$\balpha$.}
\end{definition}

\noindent Unless otherwise specified, a multi-index $\balpha$ is assumed to have the notation given in \Cref{multi-index}. For multi-indices $\balpha, \bbeta \in \N_0^n$, one defines \cite{rahman:master,slepian:master}
\begin{enumerate}
\item[1.] \emph{componentwise sum and difference}: $\balpha \pm \bbeta = (\alpha_1 \pm \beta_1, \ldots, \alpha_n \pm \beta_n)$, 
\item[2.] \emph{absolute value}: $|\balpha| := \|\balpha\|_1 = \alpha_1 + \cdots + \alpha_n$, 
which we call the \emph{order} of $\balpha$,
\item[3.] \emph{factorial}: $\balpha ! = \alpha_1! \cdots \alpha_n! = \prod_{i=1}^n \alpha_i !$, and the
\item[4.] \emph{partial derivative}: $\partial^{\balpha} = \partial_1^{\alpha_1} \cdots \partial_n^{\alpha_n}$.
\end{enumerate}
Moving forward, any $\balpha$ will denote an $n$-dimensional multi-index of order $k \ge 1$ unless otherwise specified. 

\begin{definition}[Graded lexicographic ordering]\label{glex}For multi-indices $\balpha \in \N_0^{n}$, we use a \emph{graded lexicographic ordering}; that is
\[ \balpha \le \bbeta \Longleftrightarrow  |\balpha| < |\bbeta| \text{ or } \left( |\balpha| = |\bbeta| \text{ and the \emph{leftmost} digit of } \balpha - \bbeta \text{ is positive}\right). \] 
\end{definition}
In all matrices, vectors, and sums indexed by $\balpha$, we assume a graded lexicographic ordering (in ascending order) when needed.

For multi-indices, Leibniz' formula \cite{leibniz-rule:master} for the generalized product rule is as follows:
\begin{equation}\label{leibniz-rule}
\partial^{\balpha}(fg) = \sum_{\A \le \balpha} \partial^{\A} (f) \partial^{\balpha - \A} (g).
\end{equation}
With this notation in place, we present the multi-dimensional analog of the Hermite polynomial. 

\subsection{Correlated multivariate Hermite polynomials}\label{MHP_section}
Let $(\Omega, \mathcal{A}, \mathbb{P})$ denote a complete probability space, where

\smallskip

 \begin{itemize}
\item $\Omega$ is a nonempty sample space,
\item $\mathcal{A}$ is a Borel $\sigma$-algebra on $\Omega$, and
\item $\mathbb{P}: \mathcal{A} \rightarrow [0,1]$ is a probability measure.
\end{itemize}
\medskip

Let  $\bm{\xi} \sim \mathcal{N}(\bm{0}, \Sigma)$ be an $n$-dimensional random variable with known symmetric positive definite covariance matrix $\Sigma = \mathbb{E}(\bm{\xi} \bm{\xi}^T).$  
That is, $\bm{\xi}: (\Omega, \mathcal{A}) \rightarrow (\R^n, \mathcal{B}_{\R^n})$ where $\mathcal{B}_{\R^n}$ is the Borel $\sigma$-algebra on $\R^n$. Then $\bm{\xi}$ has the joint probability density function \cite{Gauss_mult_idx_ref:master,rahman:master,xiu_thesis:master},
\begin{equation}\label{density}
\varphi: \R^n \rightarrow \R_{\ge 0} \qquad \varphi (\bm{x}; \Sigma) = \frac{1}{(2 \pi)^{n/2} |\det(\Sigma)|^{1/2} } \exp\left( - \frac{1}{2} \bm{x}^T \Sigma^{-1} \bm{x} \right).
\end{equation}

The following definition of the multivariate Hermite polynomial is a natural extension of its one-dimensional counterpart \cite{GI_book, Smith}. 

\begin{definition}[Multivariate Hermite polynomial \cite{rahman:master}]\label{mult_herm_def} Let $\bm{\xi} \sim \mathcal{N}(\bm{0}, \Sigma)$ such that $\Sigma \in \R^{n \times n}$ is symmetric positive definite (SPD) with joint density function given by Equation \cref{density}. Then for any multi-index $\balpha$ over $[n]$, the \emph{multivariate Hermite polynomial} $H_{\balpha}(\bm{\xi}; \Sigma)$ indexed by $\balpha$  is a polynomial in $\bm{\xi}$ of degree $|\balpha|$ defined as 
\[ H_{\balpha}(\bm{\xi}; \Sigma) = \begin{cases} \frac{(-1)^{|\alpha|}}{\varphi(\bm{\xi}; \Sigma)} \partial^{\balpha}(\varphi(\bm{\xi}; \Sigma)) &\text{ if } |\balpha| \ge 1 \\ 1 &\text{ if } |\balpha| = 0. \end{cases} \]
\end{definition}
We will often abbreviate a multivariate Hermite polynomial as an MHP. 

Note that since each $H_{\balpha}$ is a polynomial of order~$\oalpha$,
\begin{align}\label{deriv_too_big_0}
\partial^{\bbeta} (H_{\balpha}) = 0 \qquad \text{whenever } \oalpha < \obeta 
\end{align}
for $\balpha, \bbeta \in \N_0^n.$ Equation \cref{deriv_too_big_0} will be useful in later sections.

\Cref{MHPs} includes the first few multivariate Hermite polynomials for $\balpha \in \N_0^n$ with $n = 2$ and $n = 3$ in graded lexicographic order. We note that the polynomials in this table are not normalized, unlike those in the original reference \cite{rahman:master}. \Cref{MHPsfig} depicts the first six MHPs graphically for a given $2 \times 2$ covariance matrix. 

{\begin{table}[htbp]
\footnotesize
\caption{Multivariate Hermite polynomials (MHPs) in $\bxi \sim \mathcal{N}(\bm{0}, \Sigma)$ indexed by $\balpha \in \N_0^n$ for $n = 2, 3,$ listed in graded lexicographic ordering.  \textbf{Note that for the SPD covariance $\Sigma$ matrix, $\Sigma_{ij}^{-1}$ denotes the $(i,j)$th entry of $\Sigma^{-1}$ and \emph{not} $\frac{1}{\Sigma_{ij}}.$} The symmetry of $\Sigma^{-1}$ is applied as well.}\label{MHPs}
\renewcommand{\arraystretch}{1.1}
\begin{center}
\medskip
\begin{tabular}{ c|c|c|c }
\hline
$\balpha \in \N_0^2$ & $H_{\balpha}(\bxi)$ & $\balpha \in \N_0^3$ & $H_{\balpha}(\bxi)$ \\
\hline
$(0,0)$ & 1 & $(0,0,0)$ & 1 \\ \hline
$(1,0)$ & $\Sigma_{11}^{-1} \xi_1 + \Sigma_{12}^{-1} \xi_2 $ & $(1,0,0)$ &  $\Sigma_{11}^{-1} \xi_1 + \Sigma_{12}^{-1} \xi_2  + \Sigma_{13}^{-1} \xi_3$ \\ \hline
$(0,1)$ & $ \Sigma_{12}^{-1} \xi_1 + \Sigma_{22}^{-1} \xi_2$ & $(0,1,0)$ & $ \Sigma_{12}^{-1} \xi_1 + \Sigma_{22}^{-1} \xi_2 + \Sigma_{23}^{-1} \xi_3$ \\ \hline
 &  & $(0,0,1)$ & $\Sigma_{13}^{-1} \xi_1 + \Sigma_{23}^{-1} \xi_2 + \Sigma_{33}^{-1} \xi_3 $  \\ \hline
$(2,0)$ & $(\Sigma_{11}^{-1})^2 \xi_1^2  + 2 \Sigma_{11}^{-1} \Sigma_{12}^{-1} \xi_1 \xi_2 $  & $(2,0,0)$  & $(\Sigma_{11}^{-1})^2 \xi_1^2  + 2 \Sigma_{11}^{-1} \Sigma_{12}^{-1} \xi_1 \xi_2 $ \\ 
         & $+ (\Sigma_{12}^{-1})^2 \xi_2^2  - \Sigma_{11}^{-1}$ &  & $+ 2 \Sigma_{11}^{-1} \Sigma_{13}^{-1} \xi_1 \xi_3  + (\Sigma_{12}^{-1})^2 \xi_2^2 $ \\
         & & & $+ 2 \Sigma_{12}^{-1} \Sigma_{13}^{-1} \xi_2 \xi_3  + (\Sigma_{13}^{-1})^2 \xi_3^2 - \Sigma_{11}^{-1}$ \\ \hline
$(1,1)$ & $\Sigma_{11}^{-1} \Sigma_{12}^{-1} \xi_1^2 $  & $(1,1,0)$ & $\Sigma_{11}^{-1} \Sigma_{12}^{-1} \xi_1^2 +(\Sigma_{11}^{-1} \Sigma_{22}^{-1} + (\Sigma_{12}^{-1})^2) \xi_1 \xi_2$  \\
        & $+(\Sigma_{11}^{-1} \Sigma_{22}^{-1} + (\Sigma_{12}^{-1})^2) \xi_1 \xi_2$  & & $+(\Sigma_{11}^{-1} \Sigma_{23}^{-1} + \Sigma_{12}^{-1} \Sigma_{13}^{-1}) \xi_1 \xi_3  + \Sigma_{12}^{-1} \Sigma_{22}^{-1} \xi_2^2$ \\
        & $+ \Sigma_{12}^{-1} \Sigma_{22}^{-1} \xi_2^2 -\Sigma_{12}^{-1}$  & & $ (\Sigma_{12}^{-1} \Sigma_{23}^{-1} +  \Sigma_{13}^{-1}\Sigma_{22}^{-1})\xi_2 \xi_3 $ \\
        & & & $+ \Sigma_{13}^{-1} \Sigma_{23}^{-1} \xi_3^2 -\Sigma_{12}^{-1}$ \\ \hline
        & & $(1,0,1)$ & $\Sigma_{11}^{-1} \Sigma_{13}^{-1} \xi_1^2 + (\Sigma_{11}^{-1} \Sigma_{23}^{-1} + \Sigma_{12}^{-1} \Sigma_{13}^{-1}) \xi_1 \xi_2 $ \\
        & & & $+ (\Sigma_{11}^{-1} \Sigma_{33}^{-1} + (\Sigma_{13}^{-1})^2) \xi_1 \xi_3 + \Sigma_{12}^{-1} \Sigma_{23}^{-1} \xi_2^2 $ \\
        & & & $+ (\Sigma_{13}^{-1}\Sigma_{23}^{-1} + \Sigma_{12}^{-1} \Sigma_{33}^{-1}) \xi_2 \xi_3$ \\
        & & & $+ \Sigma_{13}^{-1} \Sigma_{33}^{-1} \xi_3^2 - \Sigma_{13}^{-1}$ \\  \hline
$(0,2)$ &$(\Sigma_{12}^{-1})^2 \xi_1^2 + 2 \Sigma_{12}^{-1} \Sigma_{22}^{-1} \xi_1 \xi_2$ & $(0,2,0)$ & $(\Sigma_{12}^{-1})^2 \xi_1^2 + 2 \Sigma_{12}^{-1} \Sigma_{22}^{-1} \xi_1 \xi_2$ \\
&          & & $+ 2 \Sigma_{12}^{-1} \Sigma_{23}^{-1} \xi_1 \xi_3 + (\Sigma_{22}^{-1})^2 \xi_2^2$  \\
& $+(\Sigma_{22}^{-1})^2 \xi_2^2 - \Sigma_{22}^{-1}$ &  & $+ 2\Sigma_{22}^{-1} \Sigma_{23}^{-1} \xi_2 \xi_3 + (\Sigma_{23}^{-1})^2 \xi_3^2 - \Sigma_{22}^{-1} $ \\  \hline
& & $(0,1,1)$ &$\Sigma_{12}^{-1} \Sigma_{13}^{-1} \xi_1^2 + (\Sigma_{12}^{-1} \Sigma_{23}^{-1} + \Sigma_{22}^{-1} \Sigma_{13}^{-1}) \xi_1 \xi_2$ \\
& &  & $+(\Sigma_{12}^{-1} \Sigma_{33}^{-1} + \Sigma_{23}^{-1} \Sigma_{13}^{-1}) \xi_1 \xi_3  + \Sigma_{22}^{-1} \Sigma_{23}^{-1} \xi_2^2$ \\
& & & $+ (\Sigma_{22}^{-1} \Sigma_{33}^{-1} + (\Sigma_{23}^{-1})^2) \xi_2 \xi_3$ \\
& & & $+ \Sigma_{23}^{-1} \Sigma_{33}^{-1} \xi_3^2 - \Sigma_{23}^{-1} $ \\  \hline
& &  $(0,0,2)$ & $(\Sigma_{13}^{-1})^2 \xi_1^2 + 2 \Sigma_{13}^{-1} \Sigma_{23}^{-1} \xi_1 \xi_2$ \\
& & & $+ 2 \Sigma_{13}^{-1} \Sigma_{33}^{-1} \xi_1 \xi_3 + (\Sigma_{23}^{-1})^2 \xi_2^2$ \\ 
& &  & $+ 2 \Sigma_{23}^{-1} \Sigma_{33}^{-1} \xi_2 \xi_3 + (\Sigma_{33}^{-1})^2 \xi_3^2 - \Sigma_{33}^{-1} $ \\  
\end{tabular}
\end{center}
\end{table}}
\begin{figure}[htbp]
\centerline{\includegraphics[scale = .6]{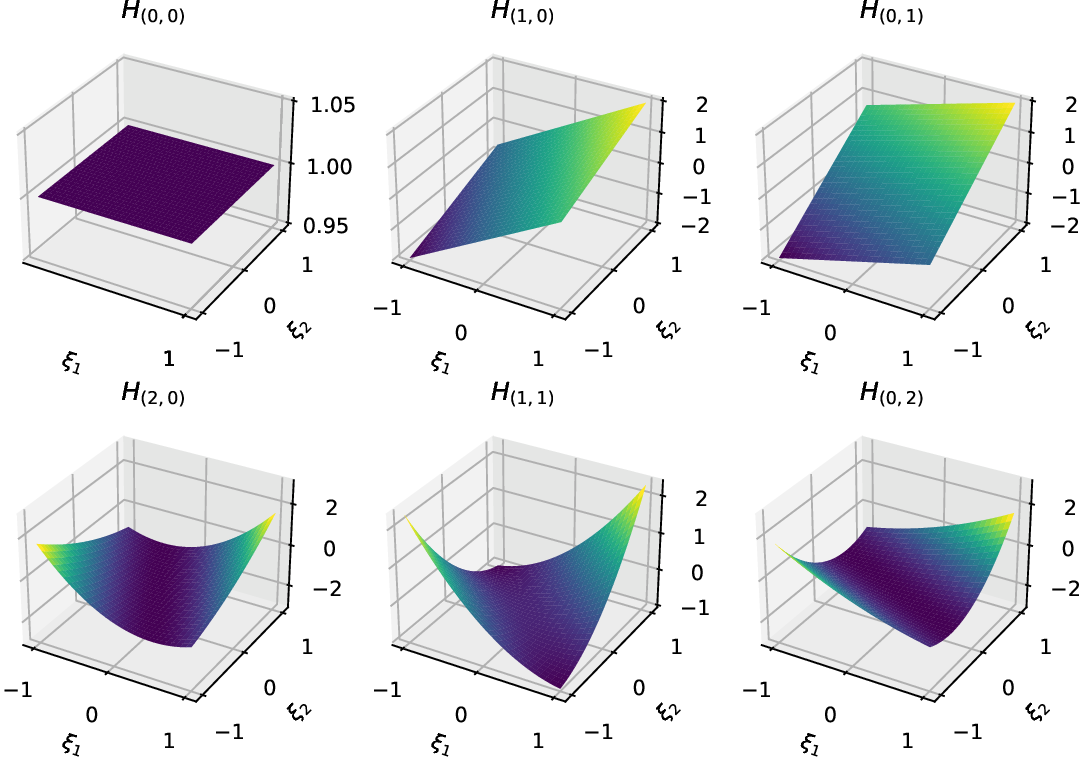}}
\caption{Multivariate Hermite polynomials $H_{\balpha}(\xi_1, \xi_2)$ for $|\balpha| \le 2$, where $\text{Var}(\xi_1) = \text{Var}(\xi_2) = 1$ and $\text{Cov}(\xi_1, \xi_2) = -1/2.$ }\label{MHPsfig}
\end{figure}
In \cite{rahman:master}, the first-order moments of the MHPs were shown to be
\begin{equation}\label{firstordermoments}
\mathbb{E}(H_{\balpha}(\bxi)) = \begin{cases} 1 &\text{ if } \balpha = (0,\ldots, 0) \\ 0 &\text{ else.} \end{cases} 
\end{equation} 
For second-order moments, Dunkl and Xu \cite{dunkl_xu_2014:master} first determined that 
\begin{equation}\label{weak_orthog}
\langle H_{\balpha}, H_{\bbeta} \rangle = \mathbb{E} (H_{\bm{\alpha}}(\bm{\xi}) H_{\bm{\beta}}(\bm{\xi})) = \int H_{\balpha}(\bm{x}) H_{\bbeta}(\bm{x}) \varphi(\bm{x}) d \bm{x} = 0  \quad \text{ if } |\bm{\alpha}| \neq |\bm{\beta}|.
\end{equation}
Rahman \cite{rahman:master} coined the condition given by \cref{weak_orthog} as \emph{weak orthogonality} in the context of polynomial chaos. 

Weak orthogonality says that $H_{\balpha}$ is orthogonal to all MHPs of different degrees but may \emph{not} be orthogonal to other MHPs of the \emph{same} $\oalpha$ degree.\footnote{In fact, $H_{\balpha}$ is usually \emph{not} orthogonal to $H_{\bbeta}$ of the same degree with $\balpha \neq \bbeta$; for example, see~\cite{mydoubprodpaper}.} For a given $n$ and $|\balpha| = M$, there are exactly $\mbinom{M + n -1}{n - 1}$ possible $\bbeta$ such that $|\balpha| = |\bbeta|$.\footnote{This follows from a standard \emph{stars and bars} counting argument \cite[Sect.~1.2]{Stanley2001:master}.}

When the components $\xi_i$ of $\bm{\xi}$ lack correlation --- which is to say that the covariance matrix $\Sigma$ is diagonal --- we can observe from \cref{MHPs} that 
\begin{align}
H_{(1,1)}(\xi_1, \xi_2) = \Sigma_{11}^{-1} \Sigma_{22}^{-1} \xi_1 \xi_2 = (\Sigma_{11}^{-1} \xi_1) (\Sigma_{22}^{-1}\xi_2) = H_{1}(\xi_1) H_{1}(\xi_2) 
\end{align} 
when $\Sigma = \text{diag}(\Sigma_{11}, \Sigma_{22})$. In fact, when $\Sigma$ is diagonal, it is well-established \cite{Feinberg, GI_book, salazar,Smith, xiu_thesis:master} that we can decompose any $H_{\balpha}$ as the product of univariate Hermite polynomials:
\begin{align}\label{1Dprod}
H_{\balpha}(\bxi) = \prod_{i=1}^n H_{\alpha_i}(\xi_i) \quad \text{for } \bxi \sim \mathcal{N}(\bm{0}, \text{diag}(\Sigma_{11}, \ldots, \Sigma_{nn})). 
\end{align} 
Equation \cref{1Dprod} is a direct result of the PDF $\varphi$ of $\bm{\xi}$ decomposing as the product of marginal densities $\varphi_i$ for each $\xi_i$ random variable when $\bm{\xi}$ has uncorrelated components. This is what is meant by constructing multivariate orthogonal polynomials via tensoriziation of univariate orthogonal polynomials. Further, the $\{H_{\balpha}\}$ are \emph{orthogonal} in this case (not just weakly orthogonal), because
\begin{align*}
\langle H_{\balpha}, H_{\bbeta} \rangle &= \int H_{\balpha}(\bm{x}) H_{\bbeta}(\bm{x}) \varphi(\bm{x}) d \bm{x} = \int \prod_{i=1}^n H_{\alpha_i}(x_i) H_{\beta_i}(x_i) \varphi_i(x_i) d x_i \\
&= \prod_{i=1}^n \int H_{\alpha_i}(x_i) H_{\beta_i}(x_i) \varphi(x_i) dx_i = \prod_{i=1}^n \langle H_{\alpha_i}, H_{\beta_i} \rangle = \balpha! \thinspace \delta_{\balpha, \bbeta}
\end{align*}
where we applied Fubini's theorem \cite{Amir:master,Baby_Rudin} and the orthogonality of the one-dimensional $H_{\alpha_i}$ polynomials. When $\bm{\xi} \sim \mathcal{N}(\bm{0}, I)$, we can also observe that
\begin{align*}
H_{(0,0)} = 1, \qquad H_{(1,0)}(\xi_1, \xi_2) = \xi_1, \qquad H_{(2,0)}(\xi_1, \xi_2) = \xi_1^2 - 1, \qquad \ldots
\end{align*} 
matching the classic one-dimensional polynomials. 

\section{PC expansion with correlated Gaussians}\label{sec_PC_corr_gauss}
The MHPs serve as the basis polynomials in the multivariate polynomial chaos expansion with correlated Gaussians established by Rahman \cite{rahman:master}. In particular, we consider the following spectral series.

\begin{definition}[Correlated multivariate polynomial chaos expansion \cite{rahman:master}]\label{MPC} Let $\bxi \sim \mathcal{N}(\bm{0}, \Sigma)$ and $y(\bm{\xi})$ be a real-valued output random variable defined on the same complete probability space $(\Omega, \mathcal{A}, \mathbb{P})$. If $y(\bxi) \in L^2(\Omega, \mathcal{A}, \mathbb{P})$ (that is, if $y(\bxi)$ is \emph{square-integrable}), we can express $y(\bm{\xi})$ via the \emph{correlated multivariate polynomial chaos expansion} as
\begin{equation*}
\bm{y}(\bm{\xi}) = \sum_{\balpha \in \N_0^n} y_{\balpha} H_{\balpha}(\bxi).
\end{equation*}
\end{definition}
\cref{MPC} \emph{generalized} the previous notion of a multivariate PC expansion by introducing basis functionals $H_{\balpha}(\bxi)$ whose formulas incorporate the non-diagonal entries of the covariance matrix of $\bxi$.\footnote{Consequently, Rahman refers to \cref{MPC} in his original paper as the \emph{generalized PCE} \cite{rahman:master}. We instead use the name \emph{correlated PC Expansion}, or \emph{CPCE}, to avoid confusion with the \emph{generalized PCE} coined by Xiu and Karniadakis \cite{xiu_thesis:master,xiu_karniadakis:master}, which is a series that still requires independent random inputs but accounts for non-Gaussian variables.} Going forward, we will sometimes refer to the series in \cref{MPC} as the \emph{correlated multivariate PCE}, the \emph{correlated PCE}, or the \emph{CPCE} for short. When referring to the multivariate PCE in \emph{independent} Gaussian parameters, we will use the phrase \emph{classical PCE}. 

When truncating to some $M$ value, we can define

\begin{equation}\label{trunc_CPCE}
y^{(M)}(\bxi) =  \sum_{\substack{\balpha \in \N_0^n \\ |\balpha| \le M}} y_{\balpha}^{(M)} H_{\balpha} (\bm{\xi}),
\end{equation}
which is a projection of $y$ onto the space $P_{M}(\bxi)$ of all polynomials  of degree up to and including $M$ in the $n$ random variables $\xi_1, \ldots, \xi_n$ \cite{Smith}. 

Then $y^{(M)}(\bxi)$ converges to the true $y(\bxi)$ in mean-square, probability, and distribution~\cite{rahman:master}. The random nodes $y_{\balpha}^{(M)}$ have the superscript $(M)$ to indicate that they solve the deterministic $M$-system (for example, see \cref{M-sys-burg}) and thus depend on $M$ (though converge to the true $y_{\balpha}$ in \cref{MPC} as $M \rightarrow \infty$). 

The number of terms in the truncated CPCE \cref{trunc_CPCE} is the same as that of the classical PCE for uncorrelated Gaussians --- the only difference is that the basis functionals $H_{\balpha}(\bxi)$ now have additional (potentially nonzero) terms whose coefficients are non-diagonal entries of $\Sigma.$ Namely, the total number of terms $n_M$ in \cref{trunc_CPCE} is given by
\begin{align}\label{n_M}
n_M:=  \big| \{ \balpha \in \N_0^n \mid 0 \le \oalpha \le M \} \big| = \frac{(M+n)!}{M!n!} = \mbinom{M + n}{n},
\end{align}
as has been noted extensively in the literature \cite{GI_book,salazar,Smith,xiu_thesis:master,xiu_karniadakis:master}. Consequently, we can observe that correlated polynomial chaos suffers from the same \emph{curse of dimensionality} as its uncorrelated counterpart, which is to say that the number of expansion terms grows rapidly when the dimensionality increases \cite{constantine_thesis:master,Smith,xiu_thesis:master}; this is overviewed in \cref{section_pc_limits}. 

As with the one-dimensional Wiener-Hermite PCE, the CPCE can be used to intrusively solve stochastic ordinary differential equations without sampling. Consider the solution $u$ to a differential equation that incorporates the stochastic quantity $\bxi$, so 

\begin{align}\label{du_dt_high_dim}
 \mathcal{F}(u, \bm{d}; \bxi) &=  0, \qquad \bm{d} \in \mathcal{D} \subset \R^p, \ \bxi \sim \mathcal{N}(\bm{0}, \Sigma).
 \end{align}
In \cref{du_dt_high_dim}, $\mathcal{F}$ is a general differential operator, usually a mix of linear and nonlinear terms. The $\mathcal{D}$ is often a nonempty, compact domain for $p$ deterministic equation parameters. 

If $u(\bm{d}; \bxi) \in L^2(\Omega, \mathcal{A}, \mathbb{P})$ for every $\bm{d} \in \mathcal{D}$, we can write
\begin{equation}\label{high_dim_PCE}
u(\bm{d}; \bm{\xi}) = \sum_{\balpha \in \N_0^n} \thinspace u_{\balpha}(\bm{d}) H_{\balpha}(\bm{\xi})
\end{equation}
where the $u_{\balpha}:\mathcal{D} \rightarrow \R$ are \emph{deterministic} coefficient functions. The truncated CPCE for the stochastic process $\{u(\bm{d};\bxi)\}_{\bm{d} \in \mathcal{D}}$ is defined as

\begin{equation}\label{high_dim_uM}
 u^{(M)}(\bm{d};\bxi) := \sum_{\substack{\balpha \in \N_0^n \\ |\balpha| \le M}} u_{\balpha}^{(M)}(\bm{d}) H_{\balpha} (\bm{\xi}).
 \end{equation}
Observe that
\begin{align*}
\mathbb{E}(u(\bm{d};\bm{\xi}) ) &= \sum_{\bm{\alpha} \in \N_0^n}  u_{\bm{\alpha}}(\bm{d}) \mathbb{E}(H_{\bm{\alpha}} (\bm{\xi})) = u_{\bm{0}}(\bm{d}) 
\end{align*}
The coefficient functions $u_{\balpha}$ decay to zero as $M \rightarrow \infty$ via Theorem 14 in \cite{rahman:master}. 

\section{Application: viscous Burgers' equation}\label{example_application}

In this section, we provide an example of using the correlated PCE to intrusively solve stochastic PDEs. We will observe how moments of the multivariate Hermite polynomials appear in the resulting deterministic systems of equations for the random modes $u_{\balpha}$.~As in the one-dimensional case, the polynomial chaos procedure for transforming equations of the form \cref{du_dt_high_dim} is to substitute $u^{(M)}$ for $u$, multiply through by an arbitrary $H_{\bbeta}$, and integrate with respect to the $d \bm{\xi}$ measure, repeating for each $H_{\bbeta}$; this is effectively projecting onto the polynomial basis.

Suppose \Cref{du_dt_high_dim} represents a variant of viscous Burgers' equation,
 \begin{equation}\label{highdim_burgers}
 \frac{\partial u}{\partial t} + u \thinspace [\bm{a}^T  \nabla_{\bm x} u] = \nu \thinspace \nabla_{\bm{x}}^2 u,  \qquad \bm{a} = (1,1)^T, \qquad \nu > 0 \text{ fixed},
 \end{equation}
with a stochastic initial condition
\begin{equation}
u(\bm{x}, 0; \bm{\xi}) = \xi_1 \sin(x_1) + \xi_2 \sin(x_2).
\end{equation}
Substituting the truncated correlated PCE into \Cref{highdim_burgers}, multiplying through by some $H_{\bm{k}}(\bm{\xi}),$ and integrating with respect to $d\bm{\xi}$ yields
\begin{align*}
 &\hphantom{=} \sum_{\substack{\bm{i} \in \N_0^n \\ |\bm{i}| \le M}} \partial_t u_{\bm{i}}^{(M)}  \langle H_{\bm{i}}, H_{\bm{k}} \rangle +  \sum_{\substack{\bm{i} \in \N_0^n \\ |\bm{i}| \le M}}  \sum_{\substack{\bm{j} \in \N_0^n \\ |\bm{j}| \le M}} u_{\bm{i}}^{(M)} (\bm{a}^T \nabla_{\bm{x}} u_{\bm{j}}^{(M)}) \langle H_{\bm{i}},  H_{\bm{j}}, H_{\bm{k}} \rangle \\
 &=  \nu  \sum_{\substack{\bm{i} \in \N_0^n \\ |\bm{i}| \le M}} \nabla_{\bm{x}}^2 u_{\bm{i}}^{(M)} \langle H_{\bm{i}}, H_{\bm{k}} \rangle 
\end{align*}
for \emph{every} $\bm{k} \in \N_0^n$ with $0 \le |\bm{k}| \le M$.

By weak orthogonality \cref{weak_orthog}, we can eliminate terms scaled by $\langle H_{\bm{i}}, H_{\bm{k}}\rangle$ whose multi-indices $\bm{i}, \bm{k}$ have different magnitudes. So, the system of couled PDEs for multivariate viscous Burgers' is
 \begin{align*}
 \sum_{\substack{\bm{i} \in \N_0^n \\ |\bm{i}| = |\bm{k}|}} \left[ \partial_t u_{\bm{i}}^{(M)} - \nu \nabla_{\bm{x}}^2 u_{\bm{i}}^{(M)}  \right] \langle H_{\bm{i}}, H_{\bm{k}} \rangle + \hspace{-1pt}  \sum_{\substack{\bm{i} \in \N_0^n \\ |\bm{i}| \le M}}  \sum_{\substack{\bm{j} \in \N_0^n \\ |\bm{j}| \le M}} u_{\bm{i}}^{(M)} (\bm{a}^T \nabla_{\bm{x}} u_{\bm{j}}^{(M)}) \langle H_{\bm{i}},  H_{\bm{j}}, H_{\bm{k}} \rangle= 0.
\end{align*}

By an identical process (Galerkin projection), the initial condition also satisfies the system
\begin{displaymath}
\sum_{|\bm{i}| \le M} u_{\bm{i}}^{(M)}(\bm{x}, 0) \langle H_{\bm{i}}, H_{\bm{k}} \rangle = \sin(x_1) \int \xi_1 H_{\bm{k}}(\bxi) \thinspace d\bxi + \sin(x_2) \int \xi_2 H_{\bm{k}}(\bxi) \thinspace d\bxi \\
\end{displaymath}
for each $0 \le |\bm{k}| \le M$. Applying a change of polynomial basis,
\begin{displaymath}
 \begin{pmatrix} \xi_1 \\ \xi_2 \end{pmatrix} = \Sigma \begin{pmatrix} H_{(1,0)}(\bxi) \\ H_{(0,1)} (\bxi) \end{pmatrix}. 
\end{displaymath}
So
\begin{align}
\langle \xi_1, H_{\bm{k}}(\bxi) \rangle &= \Sigma_{11} \langle H_{(1,0)}(\bxi), H_{\bm{k}} \rangle + \Sigma_{12} \langle H_{(0,1)}(\bxi), H_{\bm{k}} \rangle  \notag \\
&= \begin{cases} \Sigma_{11} (\Sigma^{-1})_{11} + \Sigma_{12} (\Sigma^{-1})_{12} &\text{ if } \bm{k} = (1,0) \\  \Sigma_{11} (\Sigma^{-1})_{12} + \Sigma_{12} (\Sigma^{-1})_{22} &\text{ if } \bm{k} = (0,1) \\ 0 &\text{ else.} \end{cases} \label{inner_prod_calc} \\
&= \begin{cases} 1 &\text{ if } \bm{k} = (1,0) \\ 0  &\text{ else}\end{cases}  \\
&= \delta_{(1,0), \bm{k}}.
\end{align}
To clarify the notation in \cref{inner_prod_calc}, we note that $(\Sigma^{-1})_{11}$ is entry $(1,1)$ of $\Sigma^{-1}$, which is \emph{not} necessarily $\Sigma_{11}^{-1}.$  Rather, the expansions in the first two cases of \cref{inner_prod_calc} are inner products of the first row of $\Sigma$ with the first (second, resp.) columns of $\Sigma^{-1}$. 

Similarly, $\langle \xi_2, H_{\bm{k}}(\bxi) \rangle = \delta_{(0,1), \bm{k}}$. Hence, we have the following system of initial conditions for the random modes $u_{\bm{k}}$,
\begin{align}\label{burgers_init_mult}
 \sum_{|\bm{i}| \le M} u_{\bm{i}}^{(M)}(\bm{x}, 0) \langle H_{\bm{i}}, H_{\bm{k}} \rangle &=  \sin(x_1) \delta_{(1,0), \bm{k}} + \sin(x_2) \delta_{(0,1), \bm{k}} \qquad \text{for all }  |\bm{k}| \le M.
\end{align}
Thus, in total we have a system of $n_M = \mbinom{M + n}{n}$ coupled equations, one for each $\bm{k}$ such that $|\bm{k}| \le M$:

{\small\begin{align}
&\sum_{\substack{\bm{i} \in \N_0^n \\ |\bm{i}| = |\bm{k}|}} \left[ \partial_t u_{\bm{i}}^{(M)} - \nu \nabla_{\bm{x}}^2 u_{\bm{i}}^{(M)}  \right] \langle H_{\bm{i}}, H_{\bm{k}} \rangle + \hspace{-1pt}  \sum_{\substack{\bm{i} \in \N_0^n \\ |\bm{i}| \le M}}  \sum_{\substack{\bm{j} \in \N_0^n \\ |\bm{j}| \le M}} u_{\bm{i}}^{(M)} (\bm{a}^T \nabla_{\bm{x}} u_{\bm{j}}^{(M)}) \langle H_{\bm{i}},  H_{\bm{j}}, H_{\bm{k}} \rangle= 0 \label{M-sys-burg} \\
&\sum_{|\bm{i}| \le M} u_{\bm{i}}^{(M)}(\bm{x}, 0) \langle H_{\bm{i}}, H_{\bm{k}} \rangle =  \sin(x_1) \delta_{(1,0), \bm{k}} + \sin(x_2) \delta_{(0,1), \bm{k}}. \notag
\end{align}
We will refer to the system of $n_M$ PDEs that arises from Galerkin projection as an \emph{$M$-system}, where the solution vector $\bm{u}^{(M)}$ is the collection of random mode approximations based on the truncation parameter $M$. 

Since an $M$-system is \emph{deterministic}, a computer can solve it via standard numeric techniques if the coefficients $\langle H_{\bm{i}}, H_{\bm{k}} \rangle$ and $\langle H_{\bm{i}}, H_{\bm{j}}, H_{\bm{k}} \rangle$ are known. In fact, after substituting the correlated PCE into the original PDE, \emph{any} term that is linear in $\bm{\xi}$ will generate double product coefficients in the projection process; this is not specific to the temporal derivative $\partial_t \bm{u}$ and spatial Laplacian $\nabla_{\bm{x}}^2$ terms in viscous Burgers' equation. In general, terms that are polynomial of order $N$ in $\bm{\xi}$ yield $N$th-order products after projection. Both \cite{mydoubprodpaper, rahman:master} provide expressions for $\langle H_{\bm{i}}, H_{\bm{k}} \rangle$ that can be evaluated directly. 

Prior to this work, each $\langle H_{\bm{i}}, H_{\bm{j}}, H_{\bm{k}} \rangle$ coefficient for which $0 < |\bm{i}|, |\bm{j}|, |\bm{k}| \le M$ would need to be numerically approximated. Note that we have excluded any coefficients for which one or more of the indices equals $\bm{0}$,  because they reduce to lower-order products and therefore can be computed in closed form. In particular, this means there are a total of 

\begin{align}\label{num_trip_prods_compute}
\mbinom{n_M + 1}{3},   \text{ where } n_M = \mbinom{M + n}{n},
\end{align}
triple products that would require estimation\footnote{There are $(n_M - 1)$ total options for $\balpha$  such that $0 < |\balpha| \le M$, and we are selecting $k = 3$ of them unordered and with replacement \cite{Stanley2001:master}.}. We can observe from the numerical experiments in \cref{simple_MC_conv} that even a small example (in which $M = 2$ and $n = 3$),  \Cref{num_trip_prods_compute} yields 165 individual moment computations.  Accordingly, establishing an equation to evaluate such moments is motivated naturally by the curse of dimensionality.

Furthermore, 73 of the 165 triple products are truly 0. When the triple product coefficients are approximated numerically, it can be difficult to distinguish between a genuine small physical effect in the $M$-system (that is, a PDE term whose coefficient is near zero) and numerical noise from the estimation process; this is especially important if such a perturbation nontrivially affects the system solution.  Therefore, it would be valuable to know a priori which products are 0 so that their corresponding terms in the $M$-system could be eliminated altogether with confidence.

\Cref{multivar_formulae} presents and proves results about MHPs and their derivatives that will be needed for establishing the final triple formula (\Cref{trip_prod}).  

\section{Formulas for multivariate Hermite polynomials}\label{multivar_formulae}
In \cite{mydoubprodpaper}, we established the following result.

\begin{theorem}\label{old_result}
Let $\balpha, \bbeta \in \N_0^n$ such that $\oalpha = \obeta$. Then 
\[ \mathbb{E}(\partial^{\bbeta} H_{\balpha}(\bxi)) =  \partial^{\bbeta} H_{\balpha}(\bxi) = \langle H_{\balpha}, H_{\bbeta} \rangle. \] 
\end{theorem} 

Note in \cref{old_result} that $\partial^{\bbeta} H_{\balpha}(\bxi)$ is constant (i.e., \emph{non-stochastic}), since the order of the polynomial $\oalpha$ equals the order of the derivative $\obeta$. The double product $\langle H_{\balpha}, H_{\bbeta} \rangle$ has a formula given in  \cite{mydoubprodpaper}. We will use \cref{old_result} in the proof of \Cref{sum_corollary}, which in turn will be used to establish forthcoming results. 

We begin by rewriting an MHP indexed by a sum of indices as a finite sum of lower-order Hermite polynomials multiplied by certain partial derivatives. 
\begin{theorem}[Sum of indices formula]\label{sum_of_indices}
Let $\balpha, \bbeta \in \N_0^n.$ Then
\[ H_{\balpha + \bbeta} = \sum_{\bm{q} \le \balpha} (-1)^{|\bm{q}|} \mbinom{\balpha}{\bm{q}} H_{\balpha - \bm{q}} \thinspace \partial^{\bm{q}} (H_{\bbeta}). \] 
\end{theorem}
\begin{proof}
Recall that $\varphi: \R^n \rightarrow \R_{+}$ is the density of $\bm{\xi} \sim \mathcal{N}(\bm{0}, \Sigma).$ By definition of the multivariate Hermite polynomial, 

\begin{flalign*}
H_{\balpha + \bbeta} &= (-1)^{\oalpha + \obeta} \thinspace \varphi^{-1} \thinspace \partial^{\balpha + \bbeta}(\varphi)& \text{[Def.~\ref{mult_herm_def}]}  \\ 
&= (-1)^{\oalpha + \obeta} \thinspace \varphi^{-1} \thinspace   \partial^{\balpha}(\partial^{\bbeta}(\varphi))  \\ 
&=  (-1)^{\oalpha} \thinspace \varphi^{-1} \thinspace     \partial^{\balpha}(H_{\bbeta} \thinspace \varphi) \hfill \hspace{-20pt}& \text{[$(-1)^{\obeta} \partial^{\bbeta}(\varphi)= H_{\bbeta} \thinspace \varphi$]} \\ 
&=(-1)^{\oalpha} \thinspace \varphi^{-1}  \sum_{\bm{q} \le \balpha} \mbinom{\balpha}{\bm{q}} \partial^{\bm{q}}(H_{\bbeta}) \thinspace \partial^{\balpha - \bm{q}}(\varphi) \hfill & \text{[Leibniz' (Eq.~\ref{leibniz-rule})]} \\
&= \varphi^{-1}  \sum_{\bm{q} \le \balpha} \mbinom{\balpha}{\bm{q}} \partial^{\bm{q}}(H_{\bbeta})  (-1)^{|\bm{q}|} \left[ (-1)^{\oalpha - |\bm{q}|} \partial^{\balpha - \bm{q}}(\varphi) \right]  \hspace{-150pt}\\
&= \varphi^{-1} \sum_{\bm{q} \le \balpha} \mbinom{\balpha}{\bm{q}}\partial^{\bm{q}}(H_{\bbeta})  (-1)^{|\bm{q}|}  \left[ H_{\balpha - \bm{q}} \thinspace \varphi \right]  \hfill \hspace{-90pt} & \text{[$(-1)^{\oalpha - |\bm{q}|} \partial^{\balpha - \bm{q}}(\varphi ) = H_{\balpha - \bm{q}}\thinspace \varphi$]} \\
&=  \sum_{\bm{q} \le \balpha}  (-1)^{|\bm{q}|} \mbinom{\balpha}{\bm{q}} H_{\balpha - \bm{q}} \thinspace \partial^{\bm{q}}(H_{\bbeta}). 
\end{flalign*}
\end{proof}
With \Cref{sum_of_indices}, we can prove the following lemma, which will be essential for constructing a formula for the higher-order derivative of an MHP $\partial^{\A}(H_{\bbeta})$ in \Cref{deriv_formula}.

\begin{restatable}{lemma}{esslemma}
\label{sum_corollary}Let $\A, \bbeta \in \N_0^n$ such that $|\A| \le \obeta.$ Define
\[
\mathcal{S}_1 \equiv \sum_{\substack{\B \le \bbeta \\ |\B| < |\A|}} (-1)^{|\B|} \mbinom{\bbeta}{\B} H_{\bbeta - \B} \thinspace\partial^{\B} (H_{\A})  \text{ and }  \mathcal{S}_2 := \sum_{\bm{P} < \A} (-1)^{|\bm{P}|} \mbinom{\A}{\bm{P}} H_{\A - \bm{P}} \thinspace \partial^{\bm{P}} (H_{\bbeta}).
\]
Then 
\[ \mathcal{S}_1 + (-1)^{|\A|}\presumneg \sum_{\substack{\B \le \bbeta \\ |\B| = |\A|}} \langle H_{\A}, H_{\B}\rangle  \mbinom{\bbeta}{\B} \partial^{\B}(H_{\bbeta - \B}) =  \mathcal{S}_2 + (-1)^{|\A|} \partial^{\A} (H_{\bbeta}). \] 
\end{restatable}
The proof of \Cref{sum_corollary} is given in \Cref{appendix_sec1}. By utilizing \Cref{sum_corollary} and its notation, we show inductively that the $\A$th partial derivative of $H_{\bbeta}$ for $\A, \bbeta \in \N_0^n$ can be expressed as a linear combination of lower-order MHPs.

\begin{theorem}[Higher-order derivative of MHP]\label{deriv_formula}Let $\A, \bbeta \in \N_0^n$ such that $|\A| \le \obeta.$ Then 
\[ \partial^{\A} (H_{\bbeta}) = \sum_{\substack{\B \le \bbeta \\ |\B| = |\A|}} \mbinom{\bbeta}{\B}  \langle H_{\A}, H_{\B} \rangle \thinspace H_{\bbeta - \B}. \] 
\end{theorem}

\begin{proof} 
We will induct on $|\A| \overset{\text{call}}{=} \ell.$ For $\ell = 0,$ we have 
\[   \sum_{\substack{\B \le \bbeta \\ |\B| = 0}} \mbinom{\bbeta}{\B}  \langle H_{\A}, H_{\B} \rangle H_{\bbeta - \B} = \mbinom{\bbeta}{\bm{0}} \langle H_{\bm{0}}, H_{\bm{0}} \rangle H_{\bbeta - \bm{0}} = \partial^{\bm{0}} (H_{\bbeta}) = \partial^{\A} (H_{\bbeta}), \] 
completing the base case. 

For the strong inductive step, assume for any $\bm{q} \in \N_0^n$ with $|\bm{q}| < \ell$ that 
\[ \partial^{\bm{q}}(H_{\bbeta}) = \sum_{\substack{\B \le \bbeta \\ |\B| = |\bm{q}|}} \mbinom{\bbeta}{\B}  \langle H_{\bm{q}}, H_{\B} \rangle \thinspace H_{\bbeta - \B}.\] 
Let $|\A| \le \obeta$ with $|\A| = \ell$. By \Cref{sum_corollary} with its notation,
\[ \mathcal{S}_1 + (-1)^{|\A|}\presumneg \sum_{\substack{\B \le \bbeta \\ |\B| = |\A|}} \langle H_{\A}, H_{\B}\rangle  \mbinom{\bbeta}{\B} \partial^{\B}(H_{\bbeta - \B}) =  \mathcal{S}_2 + (-1)^{|\A|} \partial^{\A} (H_{\bbeta}). \] 
In this inductive step, we will show that $\mathcal{S}_1 =  \mathcal{S}_2$. From there, we will have 
\[ \partial^{\A} (H_{\bbeta}) = \presumneg \sum_{\substack{\B \le \bbeta \\ |\B| = |\A|}} \langle H_{\A}, H_{\B}\rangle  \mbinom{\bbeta}{\B} \partial^{\B}(H_{\bbeta - \B})\] 
for $|\A| = \ell,$ as desired. Observe that
\begin{flalign*}
\mathcal{S}_1 &= \sum_{\substack{\B \le \bbeta \\ |\B| < \ell}} (-1)^{|\B|} \mbinom{\bbeta}{\B} H_{\bbeta - \B} \thinspace\partial^{\B}(H_{\A}) \hfill &\text{[def. of $\mathcal{S}_1$]} \\
&= \sum_{\substack{\B \le \bbeta \\ |\B| < \ell}} (-1)^{|\B|} \mbinom{\bbeta}{\B} H_{\bbeta - \B} \sum_{\substack{\bm{P} \le \A \\ |\bm{P}| = |\B|}} \mbinom{\A}{\bm{P}} \langle H_{\bm{P}}, H_{\B} \rangle H_{\A - \bm{P}} \hfill &\text{[induc. for $|\B| < \ell$]} 
\end{flalign*}
\begin{flalign*}
&=  \sum_{\bm{P} \le \A } \sum_{\B \le \bbeta}  \delta_{|\bm{P}| = |\B|}  \delta_{|\B| < \ell} \thinspace (-1)^{|\B|} \mbinom{\A}{\bm{P}} \mbinom{\bbeta}{\B}  \langle H_{\bm{P}}, H_{\B} \rangle H_{\A - \bm{P}} H_{\bbeta - \B}  \hspace{-60pt} \\ 
&= \sum_{\bm{P} \le \A } \sum_{\B \le \bbeta}  \delta_{|\bm{P}| = |\B|}  \delta_{|\bm{P}| < \ell} \thinspace (-1)^{|\bm{P}|} \mbinom{\A}{\bm{P}} \mbinom{\bbeta}{\B}  \langle H_{\bm{P}}, H_{\B} \rangle H_{\A - \bm{P}} H_{\bbeta - \B}  \hspace{-100pt}   &\text{[$|\bm{P}| = |\B|$]} \\
&=  \sum_{\bm{P} \le \A }  \delta_{|\bm{P}| < \ell} \thinspace (-1)^{|\bm{P}|} \mbinom{\A}{\bm{P}} H_{\A - \bm{P}} \sum_{\B \le \bbeta} \delta_{|\bm{P}| = |\B|} \mbinom{\bbeta}{\B}  \langle H_{\bm{P}}, H_{\B} \rangle H_{\bbeta - \B}  \hspace{-200pt} &\text{[rearranging]} \\
&= \sum_{\substack{\bm{P} \le \A  \\ |\bm{P}| < |\A| = \ell} }  \thinspace (-1)^{|\bm{P}|} \mbinom{\A}{\bm{P}} H_{\A - \bm{P}} \sum_{\substack{\B \le \bbeta \\ |\B| = |\bm{P}|} }\mbinom{\bbeta}{\B}  \langle H_{\bm{P}}, H_{\B} \rangle H_{\bbeta - \B} \\
&=  \sum_{\substack{\bm{P} \le \A  \\ |\bm{P}| < |\A| = \ell} }  \thinspace (-1)^{|\bm{P}|} \mbinom{\A}{\bm{P}} H_{\A - \bm{P}} \thinspace \partial^{\bm{P}} (H_{\bbeta}) \hspace{-600pt}  &\text{[induc.~for $|\bm{P}| < \ell$]}  \\
&= \mathcal{S}_2. 
\end{flalign*}

This completes the inductive step. 
\end{proof}
As a sanity check, we can recover \Cref{old_result} by taking the expectation on both sides of the main expression in \Cref{deriv_formula}. That is,
\begin{align}\label{general_deriv_expectation}
 \mathbb{E}(\partial^{\A} H_{\bbeta}(\bxi)) & = \sum_{\substack{\B \le \bbeta \\ |\B| = |\A|}} \mbinom{\bbeta}{\B} \langle H_{\A}, H_{\B} \rangle \mathbb{E}(H_{\bbeta - \B}(\bxi))  = \langle H_{\A}, H_{\bbeta} \rangle  \ \  \text{for } |\A| \le \obeta,
 \end{align}
since by \cref{firstordermoments}, the only Hermite polynomial whose expected value is nonzero is $H_{\bm{0}}$ (and $\mathbb{E}(H_0(\bxi))= 1$)  --- meaning only the double product term $\langle H_{\A}, H_{\bbeta} \rangle$ survives. However, it should be noted that \Cref{old_result} is \emph{not} merely a corollary that follows from the general derivative result (\Cref{deriv_formula}); specifically, \Cref{deriv_formula} relies on \Cref{sum_corollary}, which in turn uses the special case of $|\A| = \obeta$ given by \Cref{old_result} in its proof.

We can also note that the following result (\Cref{H_alph_deriv}) from \cite{mydoubprodpaper} is the special case of the higher-order derivative formula (\Cref{deriv_formula}) for when $|\A| = 1.$ 

\begin{lemma}\label{H_alph_deriv} Let $\balpha$ be a multi-index over $[n]$ such that $|\balpha| = k \ge 1.$ Then for any $j \in [n]$, $\partial^{\bm{e}_j} (H_{\balpha}) = \frac{\partial H_{\balpha}}{\partial \xi_j} = \sum_{i=1}^n \alpha_i \Sigma_{i j}^{-1} H_{\balpha- \bm{e_i}}. $ 
\end{lemma}

At this point, we have all the needed ingredients to prove the main theorem about the triple products (third-order moments) of Hermite polynomials in correlated Gaussian random variables.

\section{Expressions for the triple product}\label{trip_prod_sec}

Firstly, we can express the triple product as the expected value of a higher-order derivative of only two of its MHPs, as made explicit by \Cref{IBP_trip_prod_thm}. We can think of this result allowing us to ``transfer over" the index of one MHP to a derivative of the remaining factors. 

\begin{restatable}{theorem}{IBPtripprodthm} 
\label{IBP_trip_prod_thm} 
Let $\balpha, \bbeta, \bgamma \in \N_0^n$ such that $\ogamma = \text{max}(\oalpha, \obeta, \ogamma).$ Then 
\begin{displaymath}
\langle H_{\balpha}, H_{\bbeta}, H_{\bgamma} \rangle = \mathbb{E} [\partial^{\balpha} (H_{\bbeta}(\bxi) H_{\bgamma}(\bxi))].
\end{displaymath}
\end{restatable} 

The proof of \Cref{IBP_trip_prod_thm} is deferred to \Cref{appendix_sec1}. Equipped with \Cref{IBP_trip_prod_thm}, we can now derive a closed-form expression for the triple product, including conditions for which it will equal zero.

\begin{theorem}\label{trip_prod}Let $\balpha, \bbeta, \bgamma \in \N_0^n$ such that $\ogamma = \text{max}(\oalpha, \obeta, \ogamma).$ If
\begin{enumerate}
\item[(a)] $\oalpha + \obeta < \ogamma$ or 
\item[(b)] $\oalpha + \obeta + \ogamma$ is odd,
\end{enumerate}
then the triple product $\langle H_{\balpha}, H_{\bbeta}, H_{\bgamma} \rangle$ = 0. Otherwise, if $|\bgamma| \le \oalpha + |\bbeta|$ and $\oalpha + \obeta + \ogamma$ is even, define $k := \frac{\oalpha + |\bbeta| - |\bgamma|}{2}.$ Then 
\begin{equation}\label{trip_prod_eq}
\langle H_{\balpha}, H_{\bbeta}, H_{\bgamma} \rangle =  \hspace{-10pt} \sum_{\substack{\A \thinspace \le \balpha \\ \B \le \bbeta \\ |\A| = |\B| = k}} \mbinom{\balpha}{\A} \mbinom{\bbeta}{\B} \langle H_{\A}, H_{\B} \rangle \langle H_{\bm{\alpha + \beta - A - B}}, H_{\bgamma} \rangle.
\end{equation}
Note that if $k = 0,$ then \Cref{trip_prod_eq} gives that $\langle H_{\balpha}, H_{\bbeta}, H_{\bgamma} \rangle  = \langle H_{\balpha + \bbeta}, H_{\bgamma} \rangle.$
\end{theorem}
\begin{proof}
Observe that
\begin{flalign*}&\hphantom{=} \langle H_{\balpha}, H_{\bbeta}, H_{\bgamma}  \rangle \\
&= \E[\partial^{\balpha} (H_{\bbeta}(\bxi) H_{\bgamma}(\bxi))]  & \text{[by \Cref{IBP_trip_prod_thm}]} \\
&= \int_{\R^n} \partial^{\balpha}\left(H_{\bbeta}(\x) H_{\bgamma}(\x) \right)\varphi(\x) \thinspace d \x  \\ 
&= \sum_{\A \le \balpha} \mbinom{\balpha}{\A} \int \partial^{\A} (H_{\bbeta}(\bm{x})) \partial^{\balpha - \A}(H_{\bgamma}(\x) ) \varphi(\bm{x}) \thinspace d \bm{x}   \hspace{50pt} & \text{[Leibniz (Eq.~\ref{leibniz-rule})]}   \\
&= \sum_{\A \le \balpha} \sum_{\substack{\B \le \bbeta \\ |\B| = |\A|}} \mbinom{\balpha}{\A} \mbinom{\bbeta}{\B} \langle H_{\A}, H_{\B} \rangle \hspace{-4pt}  \int \hspace{-4pt} H_{\bbeta - \B}(\bm{x}) \thinspace \partial^{\balpha - \A}\left(H_{\bgamma}(\x) \right) \varphi(\bm{x}) \thinspace d \bm{x} \hspace{-25pt} & \text{[Thm.~\ref{deriv_formula}]} \\
&= \sum_{\A \le \balpha} \sum_{\substack{\B \le \bbeta \\ |\B| = |\A|}} \mbinom{\balpha}{\A} \mbinom{\bbeta}{\B} \langle H_{\A}, H_{\B} \rangle \thinspace \E[ H_{\bbeta - \B}(\bxi) \partial^{\balpha - \A}(H_{\bgamma}(\bxi) ) ]  \\ 
&= \sum_{\A \le \balpha} \sum_{\substack{\B \le \bbeta \\ |\B| = |\A|}} \mbinom{\balpha}{\A} \mbinom{\bbeta}{\B} \langle H_{\A}, H_{\B} \rangle \thinspace \E \left[ \partial^{\balpha + \bbeta - \A - \B} (H_{\bgamma}(\bxi) ) \right]  & \text{[by \Cref{simpler_expression}]} \\
&= \sum_{\A \le \balpha} \sum_{\substack{\B \le \bbeta \\ |\B| = |\A|}} \mbinom{\balpha}{\A} \mbinom{\bbeta}{\B} \langle H_{\A}, H_{\B} \rangle \langle H_{\balpha + \bbeta - \A - \B}, H_{\bgamma} \rangle. & \text{[by Thm.~\ref{old_result}]} 
\end{flalign*}
For condition (a), suppose $\oalpha + \obeta \ge \ogamma.$ Without loss of generality, assume $\oalpha \le \obeta$; the proof is symmetric in the other case. Any $\A, \B$ in the summation for which $|\balpha + \bbeta - \A - \B | \neq \ogamma$ will have their corresponding double product term $\langle H_{\balpha + \bbeta - \A - \B}, H_{\bgamma} \rangle = 0$. Observe that, since $| \A | = | \B|$ in the summation,
\begin{align}\label{nonzero_terms}
|\balpha + \bbeta - \A - \B | = \ogamma \Longleftrightarrow |\balpha + \bbeta| - 2 |\A| = \ogamma \Longleftrightarrow |\A| = \textstyle\frac{|\balpha| + |\bbeta| - |\bgamma|}{2} = k.
\end{align}
If $\oalpha + \obeta + \ogamma$ were odd (so condition (b) holds), then $\oalpha + \obeta$ and $\ogamma$ would have different parities, so $|\balpha + \bbeta| - 2 |\A| \neq \ogamma$ (accordingly, $k$ would not be an integer). Hence, given conditions (a) or (b), each $\langle H_{\balpha + \bbeta - \A - B}, H_{\bgamma} \rangle = 0$ in the expression above, so the triple product is zero.

Now suppose that $|\bgamma| \le \oalpha + |\bbeta|$ for $\ogamma = \text{max}(\oalpha, \obeta, \ogamma)$, and $\oalpha + \obeta + \ogamma$ is even, so $k := \frac{\oalpha + |\bbeta| - |\bgamma|}{2}$ is a non-negative integer. The only surviving (that is, potentially nonzero) terms are the ones for which $|\A| = |\B| = k$. In this case, we then have 
\[ \langle H_{\balpha}, H_{\bbeta}, H_{\bgamma} \rangle =  \hspace{-10pt} \sum_{\substack{\A \thinspace \le \balpha \\ \B \le \bbeta \\ |\A| = |\B| = k}} \mbinom{\balpha}{\A} \mbinom{\bbeta}{\B} \langle H_{\A}, H_{\B} \rangle \langle H_{\bm{\alpha + \beta - \A - \B}}, H_{\bgamma} \rangle \] 
as desired. \end{proof}

In the remaining part of this section, we perform some example computations with the formula given by \Cref{trip_prod}.

\begin{example}\label{ex_3} For $n = 3,$ let $\bm{\alpha} = (1,0,2), \bm{\beta} = (3,1,0),$ and $\bm{\gamma} = (1,2,2).$ Then $|\bm{\alpha}| = 3,$ $|\bm{\beta}| = 4,$ and $|\bm{\gamma}| = 5,$ so we cannot use conditions on size or parity to conclude the triple product is zero. In this case, $k := \frac{|\bm{\alpha}| + |\bm{\beta}| - |\bm{\gamma}|}{2} = 1.$

We need to find the $\bm{A} \le \bm{\alpha}$ and $\B \le \bbeta$ such that $|\bm{A}| = |\B| = 1$, so the options are $\A = \bm{e_1}, \bm{e_3}$ and $\B = \bm{e_1}, \bm{e_2}.$ The corresponding terms are computed in Table  \ref{tab_3}.

\begin{table}[H]
\caption{Computing the components of each term in the triple product formula given by Thm.~\ref{trip_prod} for $\bm{\alpha} = (1,0,2), \bm{\beta} = (3,1,0),$ and $\bm{\gamma} = (1,2,2).$}\label{tab_3}
\medskip
{\renewcommand{\arraystretch}{1.1}
\begin{center}
\begin{tabular}{c|c|c|c|c|c|c}
 & & & & & & \\[-12pt]
$\bm{A}$ & $\bm{B}$  & $\mbinom{\balpha}{\A}$ & $\mbinom{\bbeta}{\B} $ & $ \langle H_{\A}, H_{\B} \rangle$ & $ H_{\bm{\alpha + \beta - A - B}}$ & Final Term  \\[4pt] \hline 
$\bm{e_1}$ & $\bm{e_1}$ & 1 & 3 & $\Sigma_{11}^{-1}$ & $H_{(2,1,2)}$ & $3 \Sigma_{11}^{-1} \langle H_{(2,1,2)}, H_{\bm{\gamma}} \rangle$ \\[3pt] 
$\bm{e_1}$ & $\bm{e_2}$ & 1 & 1 & $\Sigma_{12}^{-1}$ & $H_{(3,0,2)}$ & $\Sigma_{12}^{-1} \langle H_{(3,0,2)}, H_{\bm{\gamma}} \rangle$ \\[3pt] 
$\bm{e_3}$ & $\bm{e_1}$ & 2 & 3 & $\Sigma_{13}^{-1}$ & $H_{(3,1,1)}$ & $6 \Sigma_{13}^{-1} \langle H_{(3,1,1)}, H_{\bm{\gamma}} \rangle$ \\[3pt] 
$\bm{e_3}$ & $\bm{e_2}$ & 2 & 1 & $\Sigma_{23}^{-1}$ & $H_{(4,0,1)}$ & $2 \Sigma_{23}^{-1} \langle H_{(4,0,1)}, H_{\bm{\gamma}} \rangle$ \\

\end{tabular}
\end{center}}
\end{table}

Therefore,
\begin{align*}
\hspace{-20pt}\langle H_{(1,0,2)}, H_{(3,1,0)}, H_{(1,2,2)} \rangle &=  \sum_{\substack{\A \le (1,0,2) \\ \B \le (3,1,0) \\ |\A| = |\B| = 1}} \mbinom{\balpha}{\A} \mbinom{\bbeta}{\B} \langle H_{\A}, H_{\B} \rangle \langle H_{\bm{\alpha + \beta - \A - \B}}, H_{\bgamma} \rangle \\
&= 3 \Sigma_{11}^{-1} \langle H_{(2,1,2)}, H_{(1,2,2)} \rangle + \Sigma_{12}^{-1} \langle H_{(3,0,2)}, H_{(1,2,2)} \rangle  \\
&+ \thinspace 6 \Sigma_{13}^{-1} \langle H_{(3,1,1)}, H_{(1,2,2)} \rangle + 2 \Sigma_{23}^{-1} \langle H_{(4,0,1)}, H_{(1,2,2)} \rangle. 
\end{align*}
\end{example}

\begin{example}\label{ex_mismatched_parity} For $n = 4,$ let $\bm{\alpha} = (1,1,1,1)$, $\bm{\beta} = (1,2,2,0)$, and $\bgamma = (1,1,2,2).$ Then $\oalpha = 4$, $\obeta = 5$, and $\ogamma = 6$, so $\oalpha + \obeta + \ogamma$ is odd. We conclude from condition (b) in \Cref{trip_prod} that $\langle H_{(1,1,1,1)}, H_{(1,2,2,0)}, H_{(1,1,2,2)} \rangle = 0.$ 
\end{example}

\begin{example}\label{ex_gamma_too_big} For $n = 4,$ let $\bm{\alpha} = (1,0,0,1), \bm{\beta} = (0,2,1,0)$,  and $\bgamma = (1,0,2,3).$ Then $\oalpha = 2$, $\obeta = 3$, $\ogamma =  6$, so $\oalpha + \obeta  < \ogamma.$  We conclude from condition (a) in \Cref{trip_prod} that $\langle H_{(1,0,0,1)}, H_{(0,2,1,0)}, H_{(1,0,2,3)} \rangle = 0.$ 
\end{example}

\subsection{Recovering the one-dimensional triple product formula}\label{sect_trip_prod_n_1}
The well-established triple product formula for $n = 1$ dimensional indices $i,j,\ell \in \N_0$ is 
\begin{align}\label{one_d_trip_form}
\langle H_i, H_j, H_\ell \rangle = \begin{cases} 0 &\text{ if } i + j + \ell \text{ is odd, or } s < \text{max}(i,j,\ell), \\
\frac{i! j! \ell!}{(s - i)! (s-j)! (s-\ell)!} &\text{ else} \end{cases} 
\end{align}
where $s := \frac{i + j + \ell}{2}$ \cite{polys:master}. 
We will recover this formula using  \Cref{trip_prod}. 

For $i,j, \ell \in \N_0$, let $\ell = \text{max}(i,j,\ell)$ without loss of generality. If $i + j + \ell$ is odd, then immediately the triple product is zero by condition (b) of \cref{trip_prod}. Now suppose $i + j + \ell$ is even, and define the integer $s := \frac{i + j + \ell}{2}.$ Then $s < \text{max}(i,j,\ell) \Longleftrightarrow
\frac{i + j + \ell}{2} < \ell \Longleftrightarrow i + j < \ell = \text{max}(i,j,\ell);$ so by condition (a), $\langle H_i, H_j, H_\ell \rangle = 0.$ 

Now consider when $i + j + \ell$ is even and $\text{max}(i,j,\ell) = \ell \le s = \frac{i + j + \ell}{2}$. Consistent with the notation in \Cref{trip_prod}, define the non-negative integer $k := \frac{i + j - \ell}{2} = s - \ell$. Then our general triple product formula gives 
\begin{flalign*}
\langle H_i, H_j, H_\ell \rangle &= \sum_{\substack{A \thinspace \le i \\ B \le j \\ |A| = |B| = k}} \mbinom{i}{A} \mbinom{j}{B} \langle H_{A}, H_{B} \rangle \langle H_{i + j - A - B}, H_{\ell} \rangle  \hspace{-50pt}
\\
&= \mbinom{i}{s - \ell} \mbinom{j}{s - \ell} \langle H_{s - \ell}, H_{s - \ell} \rangle \langle H_{i + j - 2(s - \ell)}, H_\ell \rangle \hspace{-50pt}  & \text{[only term $A = B = k$]} \\
&= \mbinom{i}{s - \ell} \mbinom{j}{s - \ell} (s - \ell)! \ell!  \hspace{-50pt} & \text{[since $i + j - 2(s - \ell) = \ell$]} \\
&= \frac{i! j!}{(s - \ell)!^2 (i - s + \ell)!  (j - s + \ell)!}  (s - \ell)! \ell!  \\
&= \frac{i! j! \ell!}{(s - \ell)! (i - s + \ell)! (j - s + \ell)!} \\
&= \frac{i! j! \ell!}{(s - \ell)! (s - j)! (s - i)!}
\end{flalign*}
where the last equality follows since $i - s + \ell = s - j$ and $j -s + \ell = s - i$. This completes the calculation. 

\section{Numerical experiments}\label{simple_MC_conv}
In this section, we numerically estimate several triple products in correlated Gaussian parameters and compare them to their predicted values from \Cref{trip_prod}. The test SPD covariance matrices are
\begin{align}\label{cov_test_matrices}
 \Sigma^{(1)} &= \begin{pmatrix} \hphantom{-}1  & - 1/2 \\ - 1/2 & \vphantom{\frac{1}{2}}\phantom{-}1 \end{pmatrix} \qquad 
 \Sigma^{(2)} =  \begin{pmatrix} 1 & \frac{1}{5}  & \frac{2}{5} \\ \frac{1}{5} & 1 & \frac{4}{5} \\ \frac{2}{5} & \frac{4}{5} & 1 \end{pmatrix}.
\end{align}
We simulate samples drawn for $\bm{\xi} \sim \mathcal{N}(\bm{0}, \Sigma)$ via \texttt{random.multivariate\_normal()} in the \texttt{numpy} package of \texttt{Python3}. Given these realizations $\bxi^{(i)}$ for $i = 1, \ldots, \EuScript{N}$ and a given integer $\EuScript{N}$, we can estimate $ \langle H_{\balpha}, H_{\bbeta}, H_{\bgamma} \rangle = \mathbb{E} (H_{\balpha}(\bxi) H_{\bbeta}(\bxi) H_{\bgamma}(\bxi))$ by the sample mean,

\begin{align}\label{S_N}
S_{\EuScript{N}} := \frac{1}{\EuScript{N}} \sum_{i=1}^{\EuScript{N}} H_{\balpha}(\bxi^{(i)})H_{\bbeta}(\bxi^{(i)}) H_{\bgamma}(\bxi^{(i)})
\end{align}
with convergence via the Central Limit Theorem at a rate of $O(\frac{1}{\sqrt{\EuScript{N}}})$ \cite{constantine_thesis:master,Amir:master,mcbook:master, Smith}. By symmetry of the triple product, we choose our index labels such that $\oalpha \le \obeta \le \ogamma.$ 

For the test covariance $\Sigma^{(1)}$, we consider all triple products $\trip$ for $\balpha, \bbeta, \bgamma \in \N_0^2$ such that $\oalpha, \obeta, \ogamma \le M$ for $M  = 2$; there are 56 total. Table SM1 in the supplement contains each theoretical value of $\trip$ (via \Cref{trip_prod}) and compares it against the sample mean $S_{\EuScript{N}}$ for $\EuScript{N} = 10^4, 10^5, \ldots, 10^8$ realizations of $\bm{\xi}^{(i)} \sim \mathcal{N}(\bm{0}, \Sigma)$. Table SM2 in the supplement displays the analogous data for all triple products with $\balpha, \bbeta, \bgamma \in \N_0^3$, $\oalpha, \obeta, \ogamma \le 2$, and test covariance $\Sigma^{(2)}$; in this case, moments involving $H_{(0,0,0)} = 1$ are excluded for brevity since these reduce to single or double products, leaving 165 remaining triple products in total (as predicted by \Cref{num_trip_prods_compute}).

\begin{figure}[htbp]
  \centering
\includegraphics[scale=.43]{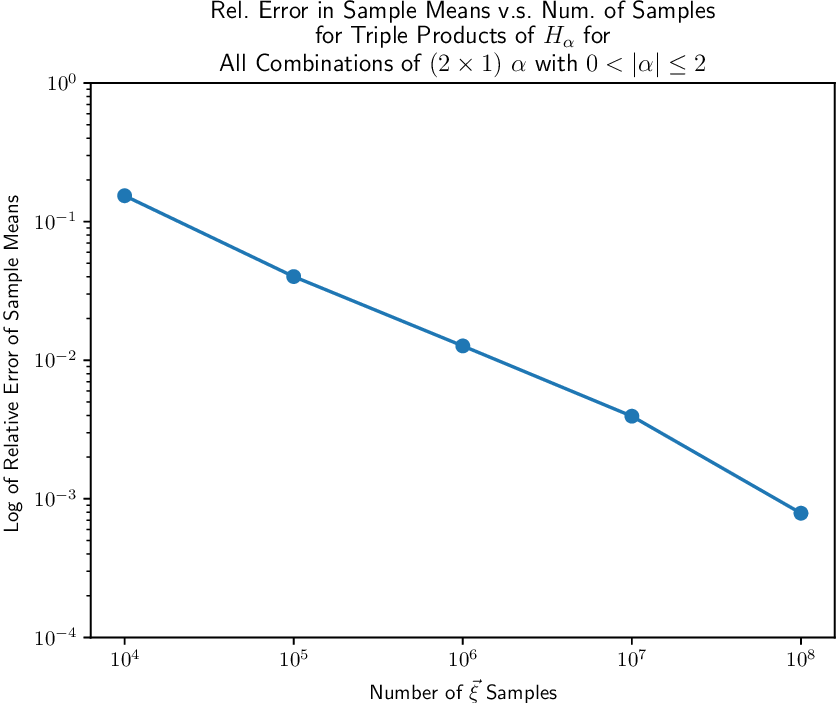}\hspace{3pt} \includegraphics[scale=.43]{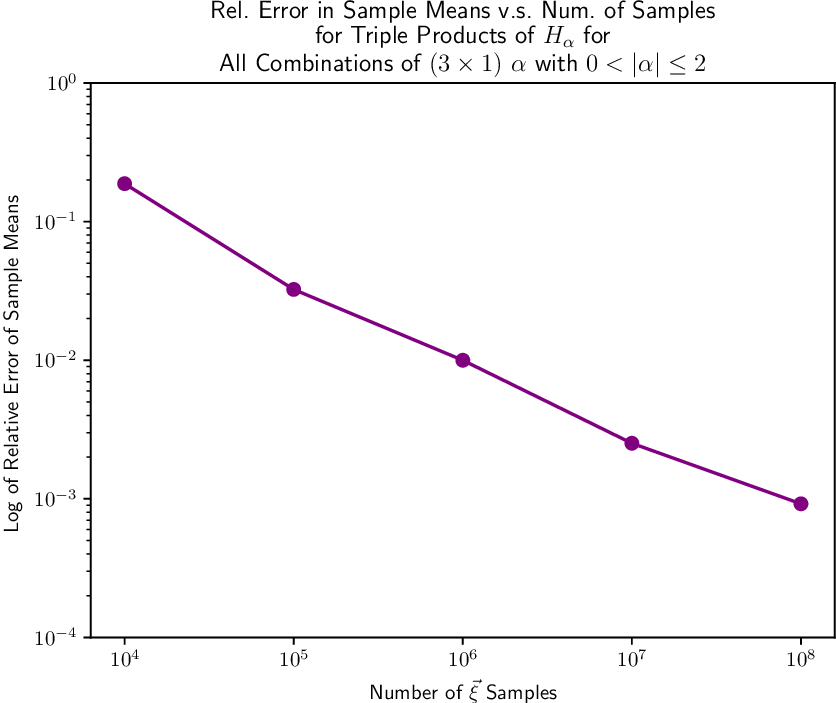}
  \caption{ Log of relative error between the vectors of triple product values estimated by simple Monte Carlo versus predicted by the formula from \Cref{trip_prod} for $\bxi = (\xi_1, \xi_2)^T \sim \mathcal{N}(\bm{0}, \Sigma^{(1)})$ (LHS) and $\bxi = (\xi_1, \xi_2, \xi_3)^T \sim \mathcal{N}(\bm{0}, \Sigma^{(2)})$ (RHS), where the covariances $\Sigma^{(1)}$  and $\Sigma^{(2)}$ are given in (\ref{cov_test_matrices}). We specifically use the vector of $\trip$ containing the 35 combinations of $\balpha, \bbeta, \bgamma \in \N_0^2$ (on the LHS) and the 165 combinations of $\balpha, \bbeta, \bgamma \in \N_0^3$ (on the RHS) such that $0 < \oalpha \le \obeta \le \ogamma \le 2$; that is, any moments with $H_{\bm{0}} = 1$ polynomials are excluded, since these reduce to lower-order products. Sample mean estimates via Monte Carlo are computed for $\EuScript{N}$ total realizations of the $\bxi$ vectors ($\EuScript{N} = 10^4, \ldots, 10^8$).}
  \label{numerics_fig_1}
\end{figure}

\Cref{numerics_fig_1} summarizes the convergence of simple Monte Carlo (MC) to the predicted triple product values. For each $\EuScript{N}$ sample count ($\EuScript{N} = 10^4, \ldots, 10^8$), we measure the relative (2-norm) error between the vector of predicted triple product values and the vector of numerical estimates for that $\EuScript{N}$. We can observe that the MC estimates conform with the predicted product values when the number of realizations of $\bxi$ is sufficiently large --- namely, $O(10^8)$ samples allow the relative error to fall to around $O(10^{-3})$ for both test correlation matrices. 

\section{Conclusions}
\label{sec:conclusions}

Stochastic Galerkin schemes provide a framework for approximating solutions to stochastic differential equations when traditional sampling-based approaches are prohibilitely costly. Broadly speaking, these methods construct a finite-degree polynomial expansion that globally approximates the solution $u(\bm{d}; \bm{\xi})$ over the parameter space. After Galerkin projection, a resultant system of deterministic PDEs is recovered whose solution is the vector of approximate random modes that appear in the polynomial expansion of $u(\bm{d}; \bm{\xi})$ \cite{constantine_primer:master, UQHandbook_IntroUQ, Smith}.   Since Galerkin methods are based on projecting the residual onto the space of approximating polynomials, their accuracy is optimal in an $L^2$ sense \cite{Smith}. When  the random vector $\bm{\xi}$ is Gaussian, this truncated series representation of $u(\bm{d}; \bm{\xi})$ is commonly called a \emph{polynomial chaos expansion} (PCE) \cite{constantine_thesis:master, stochastic-arithmetic:master,ernst:master, ghanem_and_spanos:master}. 

Given correlated random parameters, the theory behind univariate polynomial chaos does not immediately generalize. In 2017, Rahman constructed a PCE with \emph{weakly orthogonal} basis functionals consistent with the Gaussian measure, which was shown to converge to the correct limit in mean-square, probability, and distribution  \cite{rahman:master}; we refer to this series expansion as the \emph{correlated multivariate PCE} (or CPCE). A key consequence of weak orthogonality is that the \emph{$N$th products} $\mathbb{E}_{\bxi}(H_{\balpha^{(1)}}(\bxi) \cdots H_{\balpha^{(N)}}(\bxi))$ arising in Galerkin projection do not factorize into products of one-dimensional expectations. Consequently, these expectation integrals over the $n$-dimensional parameter space are typically approximated numerically  --- and doing so is not always trivial; common approaches might involve a combination of sampling methods, standard Gaussian quadrature/collocation, and sparse grid techniques. 

Rahman first provided a formula for the \emph{double products} ($N = 2$) or \emph{second-order moments} that appear during Galerkin projection \cite{rahman:master}, which was further simplified in later work \cite{mydoubprodpaper}. However, there was no closed-form expression for the $N$th products for any $N$ larger than two. Furthermore, the number of such integrals that arise grows combinatorially in both the dimension $n$ and the truncation parameter $M$ of the CPCE. At the end of \cref{example_application}, we can see explicitly how the number of third-order ($N = 3$) moments or \emph{triple products} to compute in a PDE like viscous Burgers' \cref{highdim_burgers} is already $\mbinom{n_M +1}{3}$ by \Cref{num_trip_prods_compute}, where $n_M = \mbinom{M + n}{n}$.

The primary contribution of this paper is the derivation of an explicit formula for the third-order ($N = 3$) moments or \emph{triple products} of multivariate Hermite polynomials (MHPs) with respect to correlated Gaussian distributions (\Cref{trip_prod}). To obtain this result, we establish several foundational properties of an MHP, including a decomposition for when it is indexed by a sum (\cref{sum_of_indices}) and an expression for its higher-order derivative (\Cref{deriv_formula}). We confirm that the result given by \Cref{trip_prod} reduces to the established one-dimensional formula \cref{one_d_trip_form} in \cref{sect_trip_prod_n_1}. In addition, we validate for test covariance matrices \cref{cov_test_matrices} that numerical estimates of the triple products converge to their corresponding values predicted by the new closed-form expression.

Besides presenting analogous results for $N > 3$, future directions could consider generalized polynomial chaos expansions with uncertainties represented by correlated, non-Gaussian parameters --- especially if the corresponding basis polynomials also exhibit weak orthogonality. Finding formulas for the lower-order (say, $N = 2$ or $3$) moments that arise during Galerkin projection in such contexts would be equally valuable for reducing computational costs associated with solving the resulting deterministic PDE systems. These contributions would allow polynomial chaos methods to be more readily applied to solving stochastic partial differential equations with dependent parameters.

\appendix

\section{Proofs of various results}\label{appendix_sec1}

\esslemma*
\begin{proof}By the sum of indices result (\Cref{sum_of_indices}), we can decompose $H_{\A + \bbeta}$ as
\begin{equation}\label{halphabeta-sum1}
H_{\A + \bbeta} = \sum_{\B \le \bbeta} (-1)^{|\B|} \mbinom{\bbeta}{\B}H_{\bbeta - \B} \thinspace \partial^{\B}(H_{\A})
\end{equation}
or as
\begin{equation}\label{halphabeta-sum2}
H_{\A + \bbeta} = \sum_{\bm{P} \le \A} (-1)^{|\bm{P}|} \mbinom{\A}{\bm{P}} H_{\A - \bm{P}} \thinspace \partial^{\bm{P}} (H_{\bbeta})
\end{equation}
by interchanging $\A$ and $\bbeta$. Note that $\partial^{\B} H_{\A} = 0$ when $|\B| > |\A|$ by \cref{deriv_too_big_0}. Thus, \Cref{halphabeta-sum1} becomes

\begin{align}\label{halphabeta-sum3}
H_{\A + \bbeta} &=  \sum_{\substack{\B \le \bbeta \\ |\B| \le |\A|}} (-1)^{|\B|} \mbinom{\bbeta}{\B} H_{\bbeta - \B} \thinspace \partial^{\B} (H_{\A}).
\end{align}
Applying \Cref{old_result} (i.e., $ \langle H_{\A}, H_{\B}\rangle  =\partial^{\B} H_{\A}$ when $|\A| = |\B|$) to \cref{halphabeta-sum3} gives
\begin{align}
H_{\A + \bbeta} &= \sum_{\substack{\B \le \bbeta \\ |\B| < |\A|}} (-1)^{|\B|} \mbinom{\bbeta}{\B} H_{\bbeta - \B} \thinspace\partial^{\B} H_{\A} +  (-1)^{|\A|}\presumneg \sum_{\substack{\B \le \bbeta  \\ |\B| = |\A|}} \langle H_{\A}, H_{\B}\rangle  \mbinom{\bbeta}{\B} \partial^{\B} (H_{ \bbeta - \B}) \notag \\
&= \mathcal{S}_1 + (-1)^{|\A|}\presumneg \sum_{\substack{\B \le \bbeta \\ |\B| = |\A|}} \langle H_{\A}, H_{\B}\rangle  \mbinom{\bbeta}{\B} \partial^{\B}(H_{\bbeta - \B}). \label{derivproof-eq2}  \raisetag{-16pt} 
\end{align}
Then we use a similar process, only now starting with \Cref{halphabeta-sum2}. This gives
\begin{align}
H_{\A + \bbeta} &= \sum_{\bm{P} < \A} (-1)^{|\bm{P}|} \mbinom{\A}{\bm{P}} H_{\A - \bm{P}} \thinspace \partial^{\bm{P}} H_{\bbeta} + (-1)^{|\A|} \partial^{\A} (H_{\bbeta}) \notag
\\
&= \mathcal{S}_2 + (-1)^{|\A|} \partial^{\A} H_{\bbeta}, \label{final_eq_RHS}
\end{align}
noting that $H_{\bm{0}} = 1.$ By equating \cref{derivproof-eq2} and \cref{final_eq_RHS}, the result follows. 
\end{proof}

\IBPtripprodthm*

\begin{proof} We induct on the size $\oalpha$ of the index $\balpha$. When $\oalpha = 0$, the result follows trivially since $H_{\balpha}(\bxi) = 1$. For the inductive step, let $\oalpha = \ell > 0$, and assume the desired result holds for all corresponding indices of size $\ell - 1$. Then 

\begin{flalign*}
\langle H_{\balpha}, H_{\bbeta}, H_{\bgamma} \rangle &=  \int_{\R^n} H_{\balpha}(\bm{x}) H_{\bbeta}(\bm{x}) H_{\bgamma}(\bm{x}) \varphi(\bm x) \thinspace d \bm{x} \\
&=  \int_{\R^n} \left[ H_{\balpha}(\bm{x}) \varphi(\bm{x}) \right] H_{\bbeta}(\bm{x}) H_{\bgamma}(\bm{x}) \thinspace d \bm{x}  \\
&= (-1)^{\oalpha} \int_{\R^n} \left[ \partial^{\balpha} \varphi(\bm{x}) \right] H_{\bbeta}(\bm{x}) H_{\bgamma}(\bm{x}) \thinspace d \bm{x}  \\
&=  (-1)^{\oalpha} \int_{\R^n}  \left[ \partial^{\bm{e}_j}  \partial^{\balpha - \bm{e}_j} \varphi(\bm{x}) \right] H_{\bbeta}(\bm{x}) H_{\bgamma}(\bm{x}) \thinspace d \bm{x} 
\end{flalign*}
where $j \in [n]$ is selected such that $\alpha_j \neq 0$; we know such index $j$ exists since $\oalpha = \ell > 0.$  Continuing the calculation, observe that 
\begin{flalign}
\langle H_{\balpha}, H_{\bbeta}, H_{\bgamma} \rangle &= (-1) \int_{\R^n} \partial^{\bm{e}_j}  \left( (-1)^{\oalpha - 1}   \partial^{\balpha - \bm{e}_j} \varphi(\bm{x})  \right) H_{\bbeta}(\bm{x}) H_{\bgamma}(\bm{x}) \thinspace d \bm{x} \notag \\
&= (-1)  \int_{\R^n} \partial^{\bm{e}_j}  \left( H_{\balpha - \bm{e}_j}(\bm{x}) \varphi(\bm{x})  \right) H_{\bbeta}(\bm{x}) H_{\bgamma}(\bm{x}) \thinspace d \bm{x} \notag \\
&= (-1) \int_{\R^{n-1}} \left[ \int_{\R} H_{\bbeta}(\bm{x}) H_{\bgamma}(\bm{x}) \thinspace \partial^{\bm{e}_j}\left(H_{\balpha - \bm{e}_j}(\bm{x}) \varphi(\bm{x})  \right) \thinspace dx_j \right] d\bm{x}_{-j}. \label{trip_prod_after_fubini}
\end{flalign}
where Fubini's theorem was applied and $d\bm{x}_{-j}$ denotes the Lebesgue measure on $\R^{n-1}$. Let 
\begin{displaymath}
I(x_j) = \int_{\R} H_{\bbeta}(\bm{x}) H_{\bgamma}(\bm{x}) \thinspace \partial^{\bm{e}_j}\left(H_{\balpha - \bm{e}_j}(\bm{x}) \varphi(\bm{x})  \right) \thinspace dx_j. 
\end{displaymath}
Using integration by parts in one dimension, 
\begin{align*}
I(x_j) &= H_{\bbeta}(\bm{x}) H_{\bgamma}(\bm{x}) \hspace{-2pt} \left( H_{\balpha - \bm{e}_j}(\bm{x}) \varphi(\bm{x}) \right)\hspace{-1pt} \bigg|_{\R}  \hspace{-3pt} - \int_{\R} \hspace{-2pt} \left( H_{\balpha - \bm{e}_j}(\bm{x}) \varphi(\bm{x}) \right) \partial^{\bm{e}_j} \hspace{-2pt} \left(H_{\bbeta}(\bm{x}) H_{\bgamma}(\bm{x}) \right) dx_j. 
\end{align*}
The boundary terms vanish since the MHPs have at most polynomial growth and the Gaussian density $\varphi$ decays exponentially. Thus,
\begin{flalign*}
I(x_j) &= - \int_{\R}  \left( H_{\balpha - \bm{e}_j}(\bm{x}) \varphi(\bm{x}) \right) \partial^{\bm{e}_j}  \left(H_{\bbeta}(\bm{x}) H_{\bgamma}(\bm{x}) \right) dx_j \\
&= -  \int_{\R}  \left( H_{\balpha - \bm{e}_j}(\bm{x}) \varphi(\bm{x}) \right) \left[ H_{\bbeta}(\bm{x}) \partial^{\bm{e}_j} H_{\bgamma}(\bm{x}) + H_{\bgamma}(\bm{x}) \partial^{\bm{e}_j} H_{\bbeta}(\bm{x})  \right] d x_j \\
&= -  \int_{\R}  \left( H_{\balpha - \bm{e}_j}(\bm{x}) \varphi(\bm{x}) \right) \sum_{i=1}^n \Sigma_{ij}^{-1}  \left[H_{\bbeta}(\bm{x}) \gamma_i H_{\bgamma - \bm{e}_i} (\bm{x})  + H_{\bgamma}(\bm{x})\beta_i H_{\bbeta - \bm{e}_i}(\bm{x}) \right]  d x_j
\end{flalign*}
by \Cref{H_alph_deriv}. Substituting back into \cref{trip_prod_after_fubini},
\begin{flalign*}
\langle H_{\balpha}, H_{\bbeta}, H_{\bgamma} \rangle &=  (-1) \int_{\R^{n-1}} I(x_j) \thinspace  d\bm{x}_{-j} \\
&= \int_{\R^{n-1}} \left( \int_{\R} \left( H_{\balpha - \bm{e}_j}(\bm{x}) \varphi(\bm{x}) \right) \sum_{i=1}^n \Sigma_{ij}^{-1}  \left[H_{\bbeta}(\bm{x}) \gamma_i H_{\bgamma - \bm{e}_i} (\bm{x})\right]  d x_j \right)  d\bm{x}_{-j}   \\
&+ \int_{\R^{n-1}} \left(\int_{\R} \left( H_{\balpha - \bm{e}_j}(\bm{x}) \varphi(\bm{x}) \right) \sum_{i=1}^n \Sigma_{ij}^{-1}    \left[H_{\bgamma}(\bm{x})\beta_i H_{\bbeta - \bm{e}_i}(\bm{x}) \right]  d x_j \right) d\bm{x}_{-j}   \\
&=  \sum_{i=1}^n \Sigma_{ij}^{-1} \gamma_i \int_{\R^{n-1}} \int_{\R} H_{\balpha - \bm{e}_j}(\bm{x}) H_{\bbeta}(\bm{x}) H_{\bgamma - \bm{e}_i} (\bm{x})\varphi(\bm{x}) d x_j  \thinspace d\bm{x}_{-j}  \\
&+ \sum_{i=1}^n \Sigma_{ij}^{-1} \beta_i \int_{\R^{n-1}} \int_{\R} H_{\balpha - \bm{e}_j}(\bm{x}) H_{\bbeta - \bm{e}_i}(\bm{x}) H_{\bgamma}(\bm{x}) \varphi(\bm{x}) d x_j  \thinspace d\bm{x}_{-j}.
\end{flalign*}
By reconsolidating the integrals via Fubini's theorem and using the definition of the triple product, we have
\begin{flalign*}
\langle H_{\balpha}, H_{\bbeta}, H_{\bgamma} \rangle &=\sum_{i=1}^n \Sigma_{ij}^{-1} \gamma_i \langle H_{\balpha - \bm{e}_j}, H_{\bbeta}, H_{\bgamma - \bm{e}_i} \rangle +  \sum_{i=1}^n \Sigma_{ij}^{-1} \beta_i \langle H_{\balpha - \bm{e}_j}, H_{\bbeta - \bm{e}_i}, H_{\bgamma} \rangle  \hspace{-90pt} \\
&= \sum_{i=1}^n \Sigma_{ij}^{-1} \gamma_i  \thinspace \E \left[ \partial^{\balpha - \bm{e}_j} \left( H_{\bbeta}(\bxi) H_{\bgamma - \bm{e}_i}(\bxi) \right) \right]  \\
&+  \sum_{i=1}^n \Sigma_{ij}^{-1} \beta_i  \thinspace \E \left[ \partial^{\balpha - \bm{e}_j} \left( H_{\bbeta - \bm{e}_i} (\bxi) H_{\bgamma} (\bxi)  \right) \right] \hspace{-90pt} &\text{[induc.~hypothesis]} \\
&= \E \left[ \partial^{\balpha - \bm{e}_j} \left( \sum_{i=1}^n \Sigma_{ij}^{-1} \gamma_i H_{\bbeta}(\bxi) H_{\bgamma - \bm{e}_i}(\bxi)  \right) \right]  \\
&+ \E \left[ \partial^{\balpha - \bm{e}_j} \left(   \sum_{i=1}^n \Sigma_{ij}^{-1} \beta_i H_{\bbeta - \bm{e}_i} (\bxi) H_{\bgamma} (\bxi)   \right) \right] \\
&= \E \left[ \partial^{\balpha - \bm{e}_j} \left( H_{\bbeta}(\bxi) \partial^{\bm{e}_j} H_{\bgamma}(\bxi)  \right) \right]  \\
&+ \E \left[ \partial^{\balpha - \bm{e}_j} \left( H_{\bgamma}(\bxi)\partial^{\bm{e}_j} H_{\bbeta}(\bxi) \right) \right] &\text{[by \Cref{H_alph_deriv}]} \\
&= \E \left[\partial^{\balpha - \bm{e}_j} \left( H_{\bbeta}(\bxi) \partial^{\bm{e}_j} H_{\bgamma}(\bxi)  + H_{\bgamma}(\bxi)\partial^{\bm{e}_j} H_{\bbeta}(\bxi) \right)  \right]  \\
&= \E \left[\partial^{\balpha - \bm{e}_j} \partial^{\bm{e}_j} \left( H_{\bbeta}(\bxi) H_{\bgamma}(\bxi) \right) \right] &\text{[product rule]} \\
&= \E \left[ \partial^{\balpha} (H_{\bbeta}(\bxi) H_{\bgamma}(\bxi)) \right],
\end{flalign*}
which completes the inductive step. 
\end{proof}
\begin{lemma}\label{simpler_expression}
Let $\balpha, \bbeta, \bgamma \in \N_0^n$. Then 
\[\E(H_{\balpha}(\bxi) \partial^{\bbeta} H_{\bgamma} (\bxi)) = \E(\partial^{\balpha + \bbeta}H_{\bgamma}(\bxi)).\]
\begin{proof} The argument is quite similar to those seen previously in this document, only more straight-forward. In particular, we can induct over the size of $\obeta$. Within the inductive step, we decompose $\beta$ as $\beta - \bm{e}_j + \bm{e}_j$ for some $j \in [n]$ and apply \Cref{H_alph_deriv}. 
\end{proof}
\end{lemma}


\section*{Acknowledgments}
The author would like to thank the referees for their time, feedback, and comments on this manuscript.~No AI was used for any content within or pertaining to this document. 
\newpage
\bibliographystyle{siamplain}
\bibliography{references}
\end{document}


\maketitle


In this supplement, we provide the data tables (\Cref{MC_data_table_1} and \Cref{MC_data_table_2}) corresponding to the numerical results in \cref{simple_MC_conv} of the main article. 

\begin{center}
{\tiny
\begin{longtable}{c|c|c|c|c|c|c|c|c}
\caption{Triple products $\trip$ for $\balpha, \bbeta, \bgamma \in \N_0^2$ such that $\oalpha, \obeta, \ogamma \le M$ for $M  = 2.$ The theoretical/predicted values of $\trip$ are computed by the formula in \Cref{trip_prod} and compared against the sample mean $S_{\EuScript{N}}$ for $\EuScript{N} = 10^4, 10^5, \ldots, 10^8$ realizations of $\bm{\xi}^{(i)} \sim \mathcal{N}(\bm{0}, \Sigma)$. The covariance $\Sigma$ is $2 \times 2$ with $\Sigma_{11} = \Sigma_{22}  = 1$ and  $\Sigma_{12} = \Sigma_{21} = -\frac{1}{2}$. Nonzero values are rounded to 4 significant figures. }\label{MC_data_table_1}
\medskip

\endfirsthead
\endhead
$\bm{\alpha}$ & $\bm{\beta}$  & $\bgamma$ & $\langle H_{\balpha}, H_{\bbeta}, H_{\bgamma} \rangle$ & $\EuScript{N} = 10^4$ & $\EuScript{N} = 10^5$  & $\EuScript{N} = 10^6$ & $\EuScript{N} = 10^7$ & $\EuScript{N} = 10^8$\\
& & & by Formula & & & & &  \\ \hline 
(0,0) & (0,0) & (0,0) & 1 & 1 & 1 & 1 & 1 & 1 \\
(0,0) & (0,0) & (0,1)  & 0 & -0.001812 &	-0.003172&	0.0009660 &	-0 .0004594 &	-0.00006095 \\
(0,0) & (0,0) & (1,0)  & 0 & -0.009542 &	-0.002012 &	0.0002472 &	-0.0008101 &	-0.0001342 \\ 
(0,0) &	(0,1) &	(0,1) & 1.333 &	1.304 &	1.327 &	1.331 &	1.333 &	1.333 \\
(0,0) &	(0,1) &	(1,0) & 0.6667	& 0.6662 & 0.6678 &0.6663& 0.6662 & 0.6666 \\
(0,0) &	(1,0) &	(1,0) & 1.333 & 1.368 &	1.333 &	1.335 &	1.3334 &	1.333 \\ 
(0,1) &	(0,1) &	(0,1) & 0 & -0.05911 & -0.01062 & -0.003252 & -0.002506 & -0.0008821 \\
(0,1) &	(0,1) &	(1,0) & 0 & -0.02479 & -0.006444 & -0.004386 & -0.003259	 & -0.0005723 \\
(0,1) &	(1,0) &	(1,0) & 0 & -0.006862 & 	-0.01785 &	-0.002888 &	-0.003221 &	-0.0005601 \\
(1,0) &	(1,0) &	(1,0) & 0 & -0.04782	 & -0.02083 &	-0.0005127 &	-0.004940 &	-0.0007471 \\ 
(0,0) &	(0,0) & (0,2) & 0 &	          -0.02962	&-0.006665  & -0.002527	 & -0.0005021  	& 0.0001250 \\
(0,0) &	(0,0) & (1,1) & 0 &	          -0.0004508&0.001095	  & -0.0004006 & -0.0004318	& -0.0001357 \\
(0,0) &	(0,0) & (2,0) & 0 &	          0.03487      & -0.0002084 &	0.001214 &	0.0003764	& -0.0003930 \\
(0,0) &	(0,1) & (0,2) & 0 &	          -0.05670 &-0.0063931  & -0.004540	 & -0.001894	& -0.0008008 \\
(0,0) &	(0,1) & (1,1) & 0 &	          -0.02358	   &-0.004329  & -0.005030	 & -0.002953	& -0.0005317 \\
(0,0) &	(0,1) & (2,0) & 0 &	          -0.004446	&-0.01362	  & -0.004176	 & -0.002609	& -0.0004788 \\
(0,0) &	(1,0) & (0,2) & 0 &	          -0.01207	    &-0.003761 & -0.004715	 & -0.002179	& -0.0003933 \\
(0,0) &	(1,0) & (1,1) & 0 &	          -0.0005012	&-0.01650  & -0.003053	 & -0.002681	& -0.0004706 \\
(0,0) &	(1,0) & (2,0) & 0 &	          -0.03510	&-0.01815	  & -0.0008422 & -0.003860	& -0.0005682 \\
(0,0) &	(0,2) & (0,2) & 3.556	& 3.414	&    3.466	  & 3.530	     & 3.554	    & 3.556 \\
(0,0) &	(0,2) & (1,1) & 1.778	& 1.710	&    1.754	  & 1.768	     & 1.774	    & 1.777 \\
(0,0) &	(0,2) & (2,0) & 0.8889	& 0.8419	&    0.8823	  & 0.8854     & 0.8850    & 0.8884\\
(0,0) &	(1,1) & (1,1) & 2.222	& 2.183	&    2.205	  & 2.218	     & 2.219	    & 2.222 \\
(0,0) &	(1,1) & (2,0) & 1.778		& 1.833&    1.765	  & 1.777   & 1.774    & 1.776 \\
(0,0) &	(2,0) & (2,0) & 3.556	& 3.903	    &3.541	  & 3.565	     & 3.556	    & 3.554 \\ 
(0,1) &	(0,1) & (0,2) & 3.556	& 3.374	&    3.457	  & 3.526		 & 3.553	    & 3.556 \\
(0,1) &	(0,1) & (1,1) & 1.778	& 1.709	&    1.755	  & 1.768		 & 1.773	    & 1.777\\
(0,1) &	(0,1) & (2,0) & 0.8889	& 0.8884	&    0.8820	  & 0.8870	 & 0.8855	    & 0.8879 \\
(0,1) &	(1,0) & (0,2) & 1.778	& 1.690	&    1.749	  & 1.767	     & 1.774	    & 1.777 \\
(0,1) &	(1,0) & (1,1) & 2.222	& 2.183 &    2.206	  & 2.217	     & 2.218	    & 2.222 \\
(0,1) &	(1,0) & (2,0) & 1.778	& 1.856	&    1.7652	      & 1.778	     & 1.775	    & 1.776\\
(0,1) &	(0,2) & (0,2) & 0 &         -0.2729	&    -0.001397 & -0.05745	 & -0.01581	& -0.007622\\ 
(0,1) &	(0,2) & (1,1) & 0 &     	-0.02442	&    -0.002411  & -0.02901	 & -0.01569	& -0.004676 \\
(0,1) &	(0,2) & (2,0) & 0 &     	0.06429&    -0.04859	  & -0.02505	 & -0.01352	& -0.003305 \\
(0,1) &	(1,1) & (1,1) & 0 &     	0.01179	&    -0.07372	  & -0.02868	 & -0.01620	& -0.004383 \\
(0,1) &	(1,1) & (2,0) & 0 &     	-0.01128	&    -0.08527	  & -0.02361	 & -0.01421	& -0.002523 \\
(0,1) &	(2,0) & (2,0) & 0 &     	-0.08561	&    -0.1310	  & -0.01130	 & -0.01541	& -0.002609 \\ 
(1,0) &	(1,0) & (0,2) & 0.8889	& 0.8024	&    0.8734	  & 0.8820		     & 0.8843	    & 0.8886 \\
(1,0) &	(1,0) & (1,1) & 1.778	& 1.832&    1.767	  & 1.777	     & 1.774	    & 1.776 \\
(1,0) &	(1,0) & (2,0) & 3.556	& 3.949	&    3.541	  & 3.566	     & 3.556	    & 3.553 \\ 
(1,0) &	(0,2) & (0,2) & 0 &        	-0.04612	&    -0.001659	  & -0.02575	 & -0.01405	& -0.004685 \\
(1,0) &	(0,2) & (1,1) & 0 &        	-0.003262	&    -0.05460	  & -0.02796	  & -0.01459	& -0.004110 \\
(1,0) &	(0,2) & (2,0) & 0 &         0.03255	&    -0.07015	  & -0.02527	 & -0.01080	& -0.002085 \\
(1,0) &	(1,1) & (1,1) & 0 &        	-0.04239	&    -0.08004	  & -0.02828	 & -0.01636	& -0.002918 \\
(1,0) &	(1,1) & (2,0) & 0 &         -0.06814	&    -0.1371	  & -0.01630	 & -0.01632	& -0.002869 \\
(1,0) &	(2,0) & (2,0) & 0 &         -0.6203&    -0.1926  & 0.009668	 & -0.03271	& -0.005026 \\ 
(0,2) &	(0,2) & (0,2) & 18.96	& 17.20	&    17.83	  & 18.60	     & 18.97	    & 18.95\\
(0,2) &	(0,2) & (1,1) & 9.481	& 8.217	&    9.093	  & 9.324	     & 9.444	    & 9.477 \\
(0,2) &	(0,2) & (2,0) & 4.740	& 3.912 &    4.610	  & 4.676	     & 4.703	    & 4.738 \\
(0,2) &	(1,1) & (1,1) & 8.296	& 7.267	&    8.060	  & 8.201	     & 8.255	    & 8.294 \\
(0,2) &	(1,1) & (2,0) & 5.926	& 5.401	&    5.791	  & 5.871	     & 5.888	    & 5.923 \\ 
(0,2) &	(2,0) & (2,0) & 4.740	& 4.319	&    4.591  & 4.6789     & 4.700	    & 4.7363 \\ 
(1,1) &	(1,1) & (1,1) & 7.704	& 7.214&    7.545	  & 7.642	     & 7.660	    & 7.700 \\
(1,1) &	(1,1) & (2,0) & 8.296	& 8.248	&    8.135	      & 8.244	     & 8.255	    & 8.290 \\
(1,1) &	(2,0) & (2,0) & 9.481	& 10.10	&    9.199	  & 9.432	     & 9.4328	    & 9.471 \\
(2,0) &	(2,0) & (2,0) & 18.96	& 23.51	&    18.55	  & 19.00	     & 18.92	    & 18.95\\
\end{longtable}}
\end{center}
\newpage

\begin{center}
{\tiny
\begin{longtable}{c|c|c|c|c|c|c|c|c}
\caption{Triple products $\trip$ for $\balpha, \bbeta, \bgamma \in \N_0^3$ such that $0 < \oalpha, \obeta, \ogamma \le M$ for $M  = 2.$ The theoretical/predicted values of $\trip$ are computed by the formula in \Cref{trip_prod} and compared against the sample mean $S_{\EuScript{N}}$ for $\EuScript{N} = 10^4, 10^5, \ldots, 10^8$ realizations of $\bm{\xi}^{(i)} \sim \mathcal{N}(\bm{0}, \Sigma)$. The covariance $\Sigma$ is $3 \times 3$ with $\Sigma_{11} = \Sigma_{22} = \Sigma_{33} = 1,$ $\Sigma_{12} = \Sigma_{21} = \frac{1}{5}$, $\Sigma_{13} = \Sigma_{31} = \frac{2}{5},$ and $\Sigma_{23} = \Sigma_{32} = \frac{4}{5}$.  All products including one or more $H_{\bm{0}} = 1$  are omitted, since they reduce to single or double products. Nonzero values are rounded to 4 significant figures.}\label{MC_data_table_2}
\medskip
\endfirsthead
\endhead
$\bm{\alpha}$ & $\bm{\beta}$  & $\bgamma$ & $\langle H_{\balpha}, H_{\bbeta}, H_{\bgamma} \rangle$ & $\EuScript{N} = 10^4$ & $\EuScript{N} = 10^5$  & $\EuScript{N} = 10^6$ & $\EuScript{N} = 10^7$ & $\EuScript{N} = 10^8$\\
& & & by Formula & & & & &  \\ \hline 
(0,0,1) & (0,0,1) & (0,0,1) & 0 & 0.06190 & -0.04005 & -0.006368 & -0.004591 & -0.001244 \\
(0,0,1) & (0,0,1) & (0,1,0) & 0 & -0.08195 & 0.01006 & -0.004880 & 0.001931 & -0.00003667 \\
(0,0,1) & (0,0,1) & (1,0,0) & 0 & -0.03677 & 0.01567 & 0.001010 & 0.002065 & 0.001509 \\
(0,0,1) & (0,0,1) & (0,0,2) & 22.22 & 23.09 & 22.53 & 22.40 & 22.24 & 22.23 \\
(0,0,1) & (0,0,1) & (0,1,1) & -16.67 & -17.63 & -16.84 & -16.79 & -16.69 & -16.67 \\
(0,0,1) & (0,0,1) & (0,2,0) & 12.50 & 13.51 & 12.56 & 12.58 & 12.53 & 12.50 \\
(0,0,1) & (0,0,1) & (1,0,1) & -5.556 & -5.413 & -5.639 & -5.562 & -5.558 & -5.554 \\
(0,0,1) & (0,0,1) & (1,1,0) & 4.167 & 4.142 & 4.228 & 4.174 & 4.171 & 4.166 \\
(0,0,1) & (0,0,1) & (2,0,0) & 1.389 & 1.340 & 1.384 & 1.377 & 1.390 & 1.389 \\
(0,0,1) & (0,1,0) & (0,1,0) & 0 & 0.09394 & 0.009280 & 0.007570 & 0.0006674 & 0.0007979 \\
(0,0,1) & (0,1,0) & (1,0,0) & 0 & 0.03380 & -0.005444 & 0.0009105 & -0.001840 & -0.001255 \\
(0,0,1) & (0,1,0) & (0,0,2) & -16.67 & -17.60 & -16.85 & -16.79 & -16.69 & -16.67 \\
(0,0,1) & (0,1,0) & (0,1,1) & 15.97 & 17.04 & 16.05 & 16.06 & 16.00 & 15.98 \\
(0,0,1) & (0,1,0) & (0,2,0) & -14.58 & -15.83 & -14.60 & -14.65 & -14.61 & -14.59 \\
(0,0,1) & (0,1,0) & (1,0,1) & 3.472 & 3.426 & 3.533 & 3.478 & 3.475 & 3.471 \\
(0,0,1) & (0,1,0) & (1,1,0) & -3.472 & -3.485 & -3.500 & -3.478 & -3.476 & -3.472 \\
(0,0,1) & (0,1,0) & (2,0,0) & -0.6944 & -0.7278 & -0.7068 & -0.6920 & -0.6958 & -0.6949 \\
(0,0,1) & (1,0,0) & (1,0,0) & 0 & 0.06006 & -0.0002840 & 0.001621 & 0.0004985 & -0.0004325 \\
(0,0,1) & (1,0,0) & (0,0,2) & -5.556 & -5.482 & -5.648 & -5.569 & -5.559 & -5.555 \\
(0,0,1) & (1,0,0) & (0,1,1) & 3.472 & 3.484 & 3.538 & 3.484 & 3.477 & 3.472 \\
(0,0,1) & (1,0,0) & (0,2,0) & -2.083 & -2.130 & -2.113 & -2.092 & -2.087 & -2.083 \\
(0,0,1) & (1,0,0) & (1,0,1) & 4.861 & 4.903 & 4.869 & 4.861 & 4.864 & 4.861 \\
(0,0,1) & (1,0,0) & (1,1,0) & -3.472 & -3.584 & -3.491 & -3.478 & -3.477 & -3.473 \\
(0,0,1) & (1,0,0) & (2,0,0) & -2.083 & -2.104 & -2.075 & -2.079 & -2.083 & -2.083 \\
(0,0,1) & (0,0,2) & (0,0,2) & 0 & 7.658 & -1.328 & -0.3399 & -0.1880 & 0.002532 \\
(0,0,1) & (0,0,2) & (0,1,1) & 0 & -6.382 & 0.6111 & 0.1934 & 0.1237 & -0.01395 \\
(0,0,1) & (0,0,2) & (0,2,0) & 0 & 5.365 & -0.1935 & -0.1228 & -0.06942 & 0.01872 \\
(0,0,1) & (0,0,2) & (1,0,1) & 0 & -0.2626 & 0.6732 & 0.1585 & 0.06718 & 0.02034 \\
(0,0,1) & (0,0,2) & (1,1,0) & 0 & 0.6785 & -0.3150 & -0.08467 & -0.04995 & -0.01417 \\
(0,0,1) & (0,0,2) & (2,0,0) & 0 & -0.2968 & -0.2542 & -0.05206 & -0.001004 & -0.01036 \\
(0,0,1) & (0,1,1) & (0,1,1) & 0 & 5.413 & -0.2378 & -0.1381 & -0.07008 & 0.01804 \\
(0,0,1) & (0,1,1) & (0,2,0) & 0 & -4.415 & -0.07394 & 0.1256 & 0.02112 & -0.01943 \\
(0,0,1) & (0,1,1) & (1,0,1) & 0 & 0.6640 & -0.3005 & -0.08633 & -0.05140 & -0.01523 \\
(0,0,1) & (0,1,1) & (1,1,0) & 0 & -0.9999 & 0.1409 & 0.04673 & 0.03831 & 0.01213 \\
(0,0,1) & (0,1,1) & (2,0,0) & 0 & 0.1827 & 0.1114 & 0.03001 & 0.001144 & 0.008025 \\
(0,0,1) & (0,2,0) & (0,2,0) & 0 & 3.398 & 0.3430 & -0.1677 & 0.02700 & 0.02030 \\
(0,0,1) & (0,2,0) & (1,0,1) & 0 & -0.9476 & 0.1172 & 0.04676 & 0.03859 & 0.01171 \\
(0,0,1) & (0,2,0) & (1,1,0) & 0 & 1.193 & -0.01199 & -0.03179 & -0.02977 & -0.01127 \\
(0,0,1) & (0,2,0) & (2,0,0) & 0 & -0.1294 & -0.05081 & -0.009562 & -0.0001229 & -0.005762 \\
(0,0,1) & (1,0,1) & (1,0,1) & 0 & -0.1168 & -0.2878 & -0.05046 & -0.0007869 & -0.01037 \\
(0,0,1) & (1,0,1) & (1,1,0) & 0 & 0.002716 & 0.1206 & 0.01822 & -0.0007633 & 0.007007 \\
(0,0,1) & (1,0,1) & (2,0,0) & 0 & -0.3683 & 0.09781 & -0.0009804 & -0.01017 & 0.007549 \\
(0,0,1) & (1,1,0) & (1,1,0) & 0 & 0.09867 & -0.04418 & 0.006338 & 0.004551 & -0.004507 \\
(0,0,1) & (1,1,0) & (2,0,0) & 0 & 0.2503 & -0.04796 & -0.0004651 & 0.004057 & -0.006029 \\
(0,0,1) & (2,0,0) & (2,0,0) & 0 & 0.6056 & -0.02860 & 0.02529 & 0.01309 & -0.005069 \\
(0,1,0) & (0,1,0) & (0,1,0) & 0 & -0.07858 & -0.02711 & -0.01061 & -0.003153 & -0.001360 \\
(0,1,0) & (0,1,0) & (1,0,0) & 0 & -0.05043 & 0.001298 & -0.003060 & 0.0008703 & 0.001047 \\
(0,1,0) & (0,1,0) & (0,0,2) & 12.50 & 13.56 & 12.57 & 12.59 & 12.52 & 12.50 \\
(0,1,0) & (0,1,0) & (0,1,1) & -14.58 & -15.89 & -14.61 & -14.66 & -14.61 & -14.59 \\
(0,1,0) & (0,1,0) & (0,2,0) & 17.01 & 18.54 & 17.04 & 17.08 & 17.05 & 17.02 \\
(0,1,0) & (0,1,0) & (1,0,1) & -2.083 & -2.081 & -2.109 & -2.088 & -2.085 & -2.083 \\
(0,1,0) & (0,1,0) & (1,1,0) & 2.431 & 2.513 & 2.422 & 2.438 & 2.434 & 2.430 \\
(0,1,0) & (0,1,0) & (2,0,0) & 0.3472 & 0.3884 & 0.3469 & 0.3445 & 0.3480 & 0.3478 \\
(0,1,0) & (1,0,0) & (1,0,0) & 0 & -0.05519 & -0.007570 & -0.002690 & 0.0004015 & 0.0004123 \\
(0,1,0) & (1,0,0) & (0,0,2) & 4.167 & 4.149 & 4.231 & 4.176 & 4.170 & 4.166 \\
(0,1,0) & (1,0,0) & (0,1,1) & -3.472 & -3.494 & -3.501 & -3.479 & -3.475 & -3.472 \\
(0,1,0) & (1,0,0) & (0,2,0) & 2.431 & 2.513 & 2.423 & 2.438 & 2.434 & 2.430 \\
(0,1,0) & (1,0,0) & (1,0,1) & -3.472 & -3.577 & -3.491 & -3.478 & -3.476 & -3.473 \\
(0,1,0) & (1,0,0) & (1,1,0) & 3.819 & 3.919 & 3.825 & 3.823 & 3.824 & 3.820 \\
(0,1,0) & (1,0,0) & (2,0,0) & 1.042 & 1.143 & 1.051 & 1.046 & 1.041 & 1.042 \\
(0,1,0) & (0,0,2) & (0,0,2) & 0 & -6.326 & 0.6774 & 0.2027 & 0.1280 & -0.01169 \\
(0,1,0) & (0,0,2) & (0,1,1) & 0 & 5.352 & -0.2936 & -0.1339 & -0.07654 & 0.01620 \\
(0,1,0) & (0,0,2) & (0,2,0) & 0 & -4.443 & -0.007049 & 0.1219 & 0.02928 & -0.01773 \\
(0,1,0) & (0,0,2) & (1,0,1) & 0 & 0.6590 & -0.3214 & -0.08614 & -0.05023 & -0.01458 \\
(0,1,0) & (0,0,2) & (1,1,0) & 0 & -0.9242 & 0.1501 & 0.05178 & 0.03991 & 0.01166 \\
(0,1,0) & (0,0,2) & (2,0,0) & 0 & 0.1929 & 0.1372 & 0.02630 & -0.001734 & 0.007370 \\
(0,1,0) & (0,1,1) & (0,1,1) & 0 & -4.427 & -0.01505 & 0.1155 & 0.02734 & -0.01844 \\
(0,1,0) & (0,1,1) & (0,2,0) & 0 & 3.485 & 0.2949 & -0.1549 & 0.02232 & 0.02028 \\
(0,1,0) & (0,1,1) & (1,0,1) & 0 & -0.9731 & 0.1514 & 0.04706 & 0.03965 & 0.01267 \\
(0,1,0) & (0,1,1) & (1,1,0) & 0 & 1.167 & -0.02568 & -0.03432 & -0.03278 & -0.01217 \\
(0,1,0) & (0,1,1) & (2,0,0) & 0 & -0.08810 & -0.07373 & -0.004169 & 0.002875 & -0.005544 \\
(0,1,0) & (0,2,0) & (0,2,0) & 0 & -2.369 & -0.5944 & 0.2549 & -0.07708 & -0.02314 \\
(0,1,0) & (0,2,0) & (1,0,1) & 0 & 1.187 & -0.02913 & -0.03289 & -0.03209 & -0.01195 \\
(0,1,0) & (0,2,0) & (1,1,0) & 0 & -1.491 & -0.08511 & 0.02597 & 0.02418 & 0.01235 \\
(0,1,0) & (0,2,0) & (2,0,0) & 0 & 0.07133 & 0.02692 & -0.004560 & -0.001593 & 0.004759 \\
(0,1,0) & (1,0,1) & (1,0,1) & 0 & 0.009810 & 0.1221 & 0.01808 & -0.001443 & 0.006553 \\
(0,1,0) & (1,0,1) & (1,1,0) & 0 & 0.07361 & -0.04514 & 0.004824 & 0.004513 & -0.004140 \\
(0,1,0) & (1,0,1) & (2,0,0) & 0 & 0.2380 & -0.05370 & -0.00006177 & 0.004609 & -0.005811 \\
(0,1,0) & (1,1,0) & (1,1,0) & 0 & -0.09887 & -0.02352 & -0.01779 & -0.005483 & 0.003336 \\
(0,1,0) & (1,1,0) & (2,0,0) & 0 & -0.2706 & 0.04342 & -0.01552 & -0.004662 & 0.004820 \\
(0,1,0) & (2,0,0) & (2,0,0) & 0 & -0.3063 & 0.01877 & -0.01406 & -0.006504 & 0.003025 \\
(1,0,0) & (1,0,0) & (1,0,0) & 0 & -0.1220 & -0.004610 & -0.002305 & -0.001515 & 0.0007549 \\
(1,0,0) & (1,0,0) & (0,0,2) & 1.389 & 1.443 & 1.421 & 1.395 & 1.390 & 1.389 \\
(1,0,0) & (1,0,0) & (0,1,1) & -0.6944 & -0.8153 & -0.7310 & -0.7051 & -0.6975 & -0.6949 \\
(1,0,0) & (1,0,0) & (0,2,0) & 0.3472 & 0.4613 & 0.3736 & 0.3568 & 0.3500 & 0.3479 \\
(1,0,0) & (1,0,0) & (1,0,1) & -2.083 & -2.104 & -2.080 & -2.081 & -2.083 & -2.083 \\
(1,0,0) & (1,0,0) & (1,1,0) & 1.042 & 1.154 & 1.055 & 1.047 & 1.042 & 1.042 \\
(1,0,0) & (1,0,0) & (2,0,0) & 3.125 & 3.132 & 3.091 & 3.121 & 3.127 & 3.125 \\
(1,0,0) & (0,0,2) & (0,0,2) & 0 & -0.3384 & 0.6484 & 0.1466 & 0.06187 & 0.01716 \\
(1,0,0) & (0,0,2) & (0,1,1) & 0 & 0.7298 & -0.2862 & -0.07494 & -0.04519 & -0.01162 \\
(1,0,0) & (0,0,2) & (0,2,0) & 0 & -0.9862 & 0.09992 & 0.03857 & 0.03317 & 0.008262 \\
(1,0,0) & (0,0,2) & (1,0,1) & 0 & -0.2497 & -0.2984 & -0.05326 & -0.003607 & -0.01039 \\
(1,0,0) & (0,0,2) & (1,1,0) & 0 & 0.1402 & 0.1496 & 0.02470 & -0.0008491 & 0.006781 \\
(1,0,0) & (0,0,2) & (2,0,0) & 0 & -0.06665 & 0.1112 & 0.0001739 & -0.006378 & 0.005697 \\
(1,0,0) & (0,1,1) & (0,1,1) & 0 & -1.038 & 0.1268 & 0.03957 & 0.03338 & 0.009503 \\
(1,0,0) & (0,1,1) & (0,2,0) & 0 & 1.209 & 0.0005730 & -0.02766 & -0.02599 & -0.008940 \\
(1,0,0) & (0,1,1) & (1,0,1) & 0 & 0.1082 & 0.1248 & 0.02022 & 0.0009646 & 0.006731 \\
(1,0,0) & (0,1,1) & (1,1,0) & 0 & -0.01054 & -0.06460 & 0.001226 & 0.004977 & -0.004088 \\
(1,0,0) & (0,1,1) & (2,0,0) & 0 & 0.01426 & -0.05726 & -0.001854 & 0.001046 & -0.004586 \\
(1,0,0) & (0,2,0) & (0,2,0) & 0 & -1.495 & -0.1078 & 0.02328 & 0.01899 & 0.009628 \\
(1,0,0) & (0,2,0) & (1,0,1) & 0 & -0.05371 & -0.04441 & 0.002530 & 0.002674 & -0.003656 \\
(1,0,0) & (0,2,0) & (1,1,0) & 0 & 0.02418 & -0.001987 & -0.01225 & -0.006482 & 0.002697 \\
(1,0,0) & (0,2,0) & (2,0,0) & 0 & 0.01167 & 0.05045 & -0.01300 & -0.0006495 & 0.003172 \\
(1,0,0) & (1,0,1) & (1,0,1) & 0 & -0.3643 & 0.1190 & 0.0001537 & -0.007527 & 0.009007 \\
(1,0,0) & (1,0,1) & (1,1,0) & 0 & 0.2466 & -0.06197 & -0.0008284 & 0.002269 & -0.007221 \\
(1,0,0) & (1,0,1) & (2,0,0) & 0 & 0.5791 & -0.03549 & 0.02761 & 0.01298 & -0.004879 \\
(1,0,0) & (1,1,0) & (1,1,0) & 0 & -0.3063 & 0.04898 & -0.01761 & -0.003825 & 0.005967 \\
(1,0,0) & (1,1,0) & (2,0,0) & 0 & -0.3116 & 0.01360 & -0.01462 & -0.005621 & 0.003257 \\
(1,0,0) & (2,0,0) & (2,0,0) & 0 & -0.8935 & -0.006754 & -0.02816 & -0.01855 & 0.005651 \\
(0,0,2) & (0,0,2) & (0,0,2) & 296.3 & 321.4 & 308.1 & 300.5 & 296.8 & 296.5 \\
(0,0,2) & (0,0,2) & (0,1,1) & -222.2 & -251.6 & -230.9 & -225.0 & -222.8 & -222.4 \\
(0,0,2) & (0,0,2) & (0,2,0) & 166.7 & 198.3 & 172.8 & 168.4 & 167.2 & 166.8 \\
(0,0,2) & (0,0,2) & (1,0,1) & -74.07 & -74.18 & -77.17 & -74.46 & -74.15 & -74.07 \\
(0,0,2) & (0,0,2) & (1,1,0) & 55.56 & 58.27 & 57.91 & 55.94 & 55.67 & 55.57 \\
(0,0,2) & (0,0,2) & (2,0,0) & 18.52 & 17.88 & 19.25 & 18.50 & 18.51 & 18.51 \\
(0,0,2) & (0,1,1) & (0,1,1) & 189.8 & 222.6 & 196.2 & 191.7 & 190.3 & 190.0 \\
(0,0,2) & (0,1,1) & (0,2,0) & -159.7 & -194.8 & -164.3 & -161.0 & -160.2 & -159.9 \\
(0,0,2) & (0,1,1) & (1,0,1) & 50.93 & 53.39 & 53.25 & 51.29 & 51.03 & 50.94 \\
(0,0,2) & (0,1,1) & (1,1,0) & -43.98 & -47.99 & -45.74 & -44.35 & -44.11 & -44.00 \\
(0,0,2) & (0,1,1) & (2,0,0) & -11.57 & -12.08 & -12.09 & -11.65 & -11.59 & -11.58 \\
(0,0,2) & (0,2,0) & (0,2,0) & 145.8 & 183.9 & 149.3 & 146.7 & 146.2 & 146.0 \\
(0,0,2) & (0,2,0) & (1,0,1) & -34.72 & -38.51 & -36.41 & -35.06 & -34.84 & -34.74 \\
(0,0,2) & (0,2,0) & (1,1,0) & 34.72 & 39.64 & 36.03 & 35.07 & 34.85 & 34.74 \\
(0,0,2) & (0,2,0) & (2,0,0) & 6.944 & 8.036 & 7.386 & 7.092 & 6.979 & 6.957 \\
(0,0,2) & (1,0,1) & (1,0,1) & 41.67 & 42.17 & 42.71 & 41.84 & 41.68 & 41.67 \\
(0,0,2) & (1,0,1) & (1,1,0) & -30.09 & -31.75 & -30.79 & -30.29 & -30.13 & -30.10 \\
(0,0,2) & (1,0,1) & (2,0,0) & -16.20 & -15.27 & -16.62 & -16.22 & -16.19 & -16.20 \\
(0,0,2) & (1,1,0) & (1,1,0) & 23.15 & 25.49 & 23.68 & 23.38 & 23.21 & 23.16 \\
(0,0,2) & (1,1,0) & (2,0,0) & 11.57 & 11.11 & 11.79 & 11.62 & 11.57 & 11.58 \\
(0,0,2) & (2,0,0) & (2,0,0) & 6.944 & 6.918 & 7.279 & 6.965 & 6.930 & 6.944 \\
(0,1,1) & (0,1,1) & (0,1,1) & -177.1 & -213.8 & -181.9 & -178.5 & -177.6 & -177.2 \\
(0,1,1) & (0,1,1) & (0,2,0) & 166.1 & 205.8 & 169.7 & 167.1 & 166.5 & 166.3 \\
(0,1,1) & (0,1,1) & (1,0,1) & -40.51 & -44.12 & -42.22 & -40.85 & -40.63 & -40.52 \\
(0,1,1) & (0,1,1) & (1,1,0) & 37.62 & 42.65 & 38.94 & 37.98 & 37.75 & 37.63 \\
(0,1,1) & (0,1,1) & (2,0,0) & 8.681 & 9.615 & 9.071 & 8.805 & 8.712 & 8.692 \\
(0,1,1) & (0,2,0) & (0,2,0) & -170.1 & -214.4 & -173.2 & -170.9 & -170.6 & -170.3 \\
(0,1,1) & (0,2,0) & (1,0,1) & 30.67 & 35.16 & 31.89 & 31.00 & 30.79 & 30.68 \\
(0,1,1) & (0,2,0) & (1,1,0) & -32.41 & -38.07 & -33.35 & -32.77 & -32.54 & -32.42 \\
(0,1,1) & (0,2,0) & (2,0,0) & -5.787 & -7.216 & -6.163 & -5.969 & -5.836 & -5.800 \\
(0,1,1) & (1,0,1) & (1,0,1) & -28.94 & -30.81 & -29.66 & -29.15 & -28.98 & -28.95 \\
(0,1,1) & (1,0,1) & (1,1,0) & 26.04 & 28.41 & 26.55 & 26.27 & 26.10 & 26.06 \\
(0,1,1) & (1,0,1) & (2,0,0) & 9.838 & 9.616 & 10.09 & 9.908 & 9.834 & 9.847 \\
(0,1,1) & (1,1,0) & (1,1,0) & -23.15 & -26.33 & -23.62 & -23.43 & -23.23 & -23.17 \\
(0,1,1) & (1,1,0) & (2,0,0) & -9.259 & -9.085 & -9.375 & -9.347 & -9.263 & -9.270 \\
(0,1,1) & (2,0,0) & (2,0,0) & -3.472 & -3.812 & -3.750 & -3.540 & -3.465 & -3.478 \\
(0,2,0) & (0,2,0) & (0,2,0) & 198.5 & 250.2 & 201.7 & 199.1 & 199.0 & 198.7 \\
(0,2,0) & (0,2,0) & (1,0,1) & -24.31 & -29.35 & -25.14 & -24.64 & -24.42 & -24.31 \\
(0,2,0) & (0,2,0) & (1,1,0) & 28.36 & 34.66 & 28.96 & 28.80 & 28.48 & 28.36 \\
(0,2,0) & (0,2,0) & (2,0,0) & 4.051 & 5.920 & 4.371 & 4.252 & 4.111 & 4.061 \\
(0,2,0) & (1,0,1) & (1,0,1) & 20.25 & 22.74 & 20.79 & 20.50 & 20.32 & 20.27 \\
(0,2,0) & (1,0,1) & (1,1,0) & -21.99 & -25.04 & -22.44 & -22.26 & -22.07 & -22.01 \\
(0,2,0) & (1,0,1) & (2,0,0) & -5.787 & -5.782 & -5.949 & -5.895 & -5.792 & -5.799 \\
(0,2,0) & (1,1,0) & (1,1,0) & 24.31 & 28.17 & 24.73 & 24.59 & 24.41 & 24.33 \\
(0,2,0) & (1,1,0) & (2,0,0) & 6.366 & 6.548 & 6.484 & 6.511 & 6.378 & 6.378 \\
(0,2,0) & (2,0,0) & (2,0,0) & 1.736 & 1.896 & 1.900 & 1.805 & 1.732 & 1.742 \\
(1,0,1) & (1,0,1) & (1,0,1) & -21.99 & -20.88 & -22.49 & -22.01 & -21.97 & -21.99 \\
(1,0,1) & (1,0,1) & (1,1,0) & 14.47 & 14.14 & 14.77 & 14.53 & 14.46 & 14.47 \\
(1,0,1) & (1,0,1) & (2,0,0) & 15.63 & 15.54 & 15.88 & 15.63 & 15.62 & 15.62 \\
(1,0,1) & (1,1,0) & (1,1,0) & -11.57 & -11.55 & -11.74 & -11.68 & -11.58 & -11.58 \\
(1,0,1) & (1,1,0) & (2,0,0) & -10.42 & -10.73 & -10.64 & -10.48 & -10.41 & -10.42 \\
(1,0,1) & (2,0,0) & (2,0,0) & -10.42 & -10.92 & -10.60 & -10.38 & -10.39 & -10.41 \\
(1,1,0) & (1,1,0) & (1,1,0) & 9.259 & 9.766 & 9.399 & 9.424 & 9.270 & 9.271 \\
(1,1,0) & (1,1,0) & (2,0,0) & 10.42 & 10.58 & 10.50 & 10.48 & 10.42 & 10.42 \\
(1,1,0) & (2,0,0) & (2,0,0) & 5.208 & 6.177 & 5.473 & 5.270 & 5.191 & 5.215 \\
(2,0,0) & (2,0,0) & (2,0,0) & 15.63 & 15.98 & 15.52 & 15.57 & 15.59 & 15.62 \\
\end{longtable}}
\end{center}

\bibliographystyle{siamplain}
\bibliography{references}